\documentclass[11pt]{article}

\usepackage{latexsym}
\usepackage{amssymb}
\usepackage{amsthm}
\usepackage{amscd}
\usepackage{amsmath}
\usepackage{mathabx}
\usepackage{mathrsfs}
\usepackage{tikz}

\usepgflibrary{arrows}

\newtheorem{theorem}{Theorem}[section]

\newtheorem{lemma}[theorem]{Lemma}

\newtheorem{proposition}[theorem]{Proposition}
\newtheorem{corollary}[theorem]{Corollary}

\theoremstyle{definition}
\newtheorem*{example}{Example}
\newtheorem*{examples}{Examples}
\newtheorem*{remark}{Remark}

\numberwithin{equation}{section}

\newcommand{\CC}{\mathbb{C}} 
\newcommand{\FF}{\mathbb{F}}  

\newcommand{\QQ}{\mathbb{Q}}  
\newcommand{\ZZ}{\mathbb{Z}}

\newcommand{\cA}{\mathcal{A}}
\newcommand{\cB}{\mathcal{B}}

\newcommand{\cK}{\mathcal{K}}

\newcommand{\cP}{\mathcal{P}}

\newcommand{\cT}{\mathcal{T}}

\newcommand{\cX}{\mathcal{X}}

\newcommand{\fkgl}{\mathfrak{gl}}
\newcommand{\fkut}{\mathfrak{ut}}
\newcommand{\fklt}{\mathfrak{lt}}

\newcommand{\fkb}{\mathfrak{b}}

\newcommand{\scS}{\mathscr{S}}
\newcommand{\scM}{\mathscr{M}}

\newcommand{\tr}{\mathrm{tr}}

\newcommand{\GL}{\mathrm{GL}}
  
\newcommand{\Res}{\mathrm{Res}}

\newcommand{\Irr}{\mathrm{Irr}}

\newcommand{\wt}{\mathrm{wt}}

\newcommand{\Id}{\mathrm{Id}}

\newcommand{\UT}{\mathrm{UT}}

\newcommand{\nst}{\mathrm{nst}}
\newcommand{\cf}{\mathrm{cf}}

\newcommand{\scf}{\mathrm{scf}}

\newcommand{\dd}{\displaystyle}
\newcommand{\scs}{\scriptstyle}
\newcommand{\scscs}{\scriptscriptstyle}

\newcommand{\vphi}{\varphi}

\newcommand{\spanning}{\textnormal{-span}}
\newcommand{\One}{{1\hspace{-.14cm} 1}}

\newcommand{\larc}[1]{\hspace{-.4ex}\overset{#1}{\frown}\hspace{-.4ex}}
\newcommand{\slarc}[1]{\overset{#1}{\frown}}
\def\adots{\mathinner{\mkern2mu\raise0pt\hbox{.}  
\mkern2mu\raise4pt\hbox{.}\mkern1mu
\raise7pt\vbox{\kern7pt\hbox{.}}\mkern1mu}}

\newcommand{\bl}{\mathrm{bl}}
\newcommand{\uncr}[1]{\underset{\asymp}{#1}}
\newcommand{\crs}{\mathrm{crs}}
\newcommand{\Pqbin}[3]{\genfrac{[}{]}{0pt}{}{#1}{#2}_{#3}}

\newcommand{\epl}[1]{\underset{\curvearrowbotleft}{#1}}
\newcommand{\epr}[1]{\underset{\curvearrowbotright}{#1}}

\newcommand{\lact}{\ {\scs\circlearrowright}\ }
\newcommand{\ract}{\ {\scs\circlearrowleft}\ }

\makeatletter
\renewcommand{\@makefnmark}{\mbox{\textsuperscript{}}}
\makeatother

\allowdisplaybreaks[1]

\begin{document}

\title{Restrictions of rainbow supercharacters}
\author{Daniel Bragg\\ Department of Mathematics\\ University of Washington Seattle\\ \textsf{braggdan@uw.edu} \and Nathaniel Thiem\\ Department of Mathematics\\ University of Colorado \textbf{Boulder}\\
 \textsf{thiemn@colorado.edu}}

\date{}

\maketitle

\begin{abstract} The maximal subgroup of unipotent upper-triangular matrices of the finite general linear groups are a fundamental family of $p$-groups.  Their representation theory is well-known to be wild, but there is a standard supercharacter theory, replacing irreducible representations by super-representations, that gives us some control over its representation theory.  While this theory has a beautiful underlying combinatorics built on set partitions, the structure constants of restricted super-representations remain mysterious.  This paper proposes a new approach to solving the restriction problem by constructing natural intermediate modules that help ``factor" the computation of the structure constants.  We illustrate the technique by solving the problem completely in the case of rainbow supercharacters (and some generalizations).  Along the way we introduce a new $q$-analogue of the binomial coefficients that depend on an underlying poset.
\end{abstract}

\section{Introduction\protect\footnote{MSC 2010: 20C33, 05E10}}
 Let $N$ be a set with a total order so that we can construct the group $\GL_N(\FF_q)$ of invertible matrices with rows and columns indexed by $N$ and entries in the finite field $\FF_q$ with $q$ elements.  The supercharacter theory of the finite unitriangular groups 
$$\UT_N=\{u\in \GL_N(\FF_q)\mid (u-\Id_N)_{ij}\neq 0\text{ implies } i<j\},$$
where $\Id_N$ is the multiplicative identity of $\GL_N(\FF_q)$, has developed into a rich combinatorics based on set partitions.  In particular, \cite{AIM} showed that --- taken as a family --- they give a representation theoretic realization of the Hopf algebra of symmetric functions in noncommuting variables (also studied in \cite{RS}, for example), where the product comes from inflation and the coproduct from restriction.     Thus, the representation theory of unipotent $p$-groups gives a noncommuting analogue to the classical combinatorial representation theory of the symmetric groups.
The supercharacters of these groups give a new  basis for this Hopf algebra.

One obstruction to making use of this connection is that while the inflation functor is straightforward for supercharacters, the restriction functor is still somewhat mysterious.  That is, given a subset $K\subseteq N$, we want to decompose a supercharacter of $\UT_N$ as a linear combination of supercharacters of the subgroup
$$\UT_K=\{u\in \UT_N\mid (u-\Id_N)_{ij}\neq 0\text{ implies } i,j\in K\}\subseteq\UT_N.$$ 
The paper \cite{Th10} gives an iterative algorithm for computing restrictions of supercharacters, but this gives us little information about the coefficients that occur (though it does imply that they will be polynomial in the size of the underlying field $q$).  As a preliminary step, \cite{LT11} uses matchings in bipartite graphs to give  a combinatorial characterization of when such a coefficient is nonzero; however, only a small set of examples have a direct computation of the coefficients.  

The supercharacters of $\UT_N$ are indexed by set partitions of the set $N$.  In this subject, it seems preferable to view set partitions as a set of pairs, as follows.  Given a set partition $\bl(\lambda)$ of $N$, we can store the block information as a set of pairs 
$$\lambda=\left\{(i,j)\ \Big|\ \begin{array}{@{\ }l@{}}  \text{$i<j$ with $i,j$ in the same block of $\bl(\lambda)$,}\\ \text{$i<j'<j$ implies $j'$ is in a different block}\end{array}\right\}.$$
The supercharacter $\chi^\lambda$ corresponding $\lambda$ has a convenient factorization
$$\chi^\lambda=\bigodot_{(i,j)\in \lambda} \chi^{(i,j)},$$
where $\odot$ is the pointwise product on functions.  These pairs then fall into three cases:
\begin{description}
\item[Case 1.]  $(i,j)\in \lambda$ satisfies $|\{i,j\}\cap K|=0$,
\item[Case 2.]  $(i,j)\in \lambda$ satisfies $|\{i,j\}\cap K|=1$,
\item[Case 3.]  $(i,j)\in \lambda$ satisfies $|\{i,j\}\cap K|=2$.
\end{description}
It follows from \cite{Th10} that Case 3 is easy to deal with (the pair restricts to the same pair, up to power of $q$).  Case 2 seems reasonably manageable (the one endpoint in $K$ acts as an anchor).  Case 1 is the most unpredictable, and it is the goal of this paper to begin developing a theory to tackle this case.

In particular, a natural first case to consider is the case where 
$$N=\{1,2,\ldots,2m+k\}=\{1,2,\ldots, m\} \sqcup K\sqcup \{m+k+1,\ldots, 2m+k\},$$
and $\lambda=\{(1,2m+k),(2,2m+k-1),\ldots, (m,m+k+1)\}$, known as a \emph{rainbow}  set partition (see the picture (\ref{MultiRainbowPicture}) at $\ell=1$ below).  In this case, it follows easily from \cite{LT11} that  
$$\Res_{\UT_K}^{\UT_N} (\chi^\lambda)=\sum_{\text{set partition $\nu$ of $K$}\atop |\nu|\leq m} c_\nu \chi^\nu,$$
where all the $c_\nu$ are nonzero.  However, trying to compute these coefficients with the techniques of \cite{Th10} is prohibitively complicated.

This paper proposes an alternative approach by observing that as a class function of $\UT_K$, the rainbow character $\chi^\lambda$ does not distinguish well between superclasses.  We can therefore find a set of characters $\{\psi_K^0,\psi_K^1,\ldots, \psi_K^{|K|}\}$ with natural and explicit modules that span a subspace of superclass functions in which the restrictions land.   Conveniently, the rainbow characters decompose nicely in this space, and these new characters also decompose nicely in terms of supercharacters.  Combining these two decompositions gives coefficients in the case of the rainbow supercharacters.  

The decompositions of the $\psi_K^{(j)}$ into supercharacters leads into an apparently new variation of $q$-binomial coefficients that depend on a poset (the total order on a set gives the usual $q$-binomial).  We define these combinatorial gadgets, and explore some of their properties.  

We expand the rainbow program further and consider multi-rainbows or set partitions of the form
\begin{equation}\label{MultiRainbowPicture}
\begin{tikzpicture}[scale=.3,baseline=0]
	\foreach \x in {1,2,4,5,6,8,9,10,12,13,14,16,
				24,26,27,28,30,31,32,34,35,36,38,39}
		\node (\x) at (\x,0) [inner sep=-1pt] {$\scs \bullet$};
	\foreach \x in {1,2,4,5,6,8,9,10,12,13,14,16}
		\node at (\x,.75) {$\scscs \x$};
	\foreach \x/\y in {2/38,4/36,9/31,10/30,12/28}
		\draw (\x) to [out=-45,in=-135] (\y);
	\draw (1) .. controls (11,-11) and (28,-11) .. (39);
	\foreach \x in {3,7,11,15,25,29,33,37}
		\node at (\x,0) {$\scs\cdots$};
	\node at (20.5,0) {$\cdots$};
	\draw[|-|] (5,1.5) --  (8,1.5);
	\node at (6.5,2) {$\scscs K_1$};
	\draw[|-|] (13,1.5) --  (16,1.5);
	\node at (14.5,2) {$\scscs K_2$};
	\draw[|-|] (24,.75) --  (27,.75);
	\node at (25.5,1.5) {$\scscs K_{\ell-1}$};
	\draw[|-|] (32,.75) --  (35,.75);
	\node at (33.5,1.5) {$\scscs K_\ell$};
\end{tikzpicture}\ ,
\end{equation}
where $K=K_1\cup K_2\cup\cdots\cup K_\ell\subseteq N$.  We study the case $\ell=3$ in the most depth. In this case, we mix metaphors and develop a notion of ``peel" module.  In the general $\ell$ case, the decomposition will be in terms of products of peel modules (nested in a way reminiscent of an onion).  

The paper is organized as follows.  In Section \ref{SectionPreliminaries}, we introduce our notation for set partitions, the relevant supercharacter theory of $\UT_N$, and their corresponding modules.  In Section \ref{SectionRainbow} we study the case of rainbow characters; we introduce a notion of poset $q$-binomial and define the modules corresponding to the characters $\{\psi_K^j\}$.   We decompose the rainbow modules in terms of these new modules, and then decompose further into supercharacters.  Section \ref{SectionOnionPeel} then explores the case where $\ell=3$ in (\ref{MultiRainbowPicture}), and indicates how one might generalize to arbitrary $\ell$.  

\vspace{.5cm}

\noindent \textbf{Some remarks on notation.} This paper has made several somewhat nonstandard notational choices.  The first choice is to study $\UT_N$ for a set $N$ rather than $\UT_{|N|}$ for a number $|N|$. There are two reasons for this choice.  First, the representation theory of $\UT_N$ seems most natural in the context of Hopf monoids in species, which associate sets $N$ to groups \cite{ABT}.  Second, for restriction problems it is critical to see the exact subset $K\subseteq N$.  That is, up to isomorphism the subgroups only depend on the cardinality of the subset, but the restriction problem depends on the actual embedding.

The second choice is to express the main results as isomorphisms of modules, rather than character formulas.  We do this to stress that all our characters have explicit associated modules, and these modules actually help in finding the right decompositions. 

\vspace{.5cm}

\noindent \textbf{Acknowledgements.}  Bragg was supported in part by DMS-0854893 and Thiem was supported in part by H98230-13-1-0203.  The authors would also like to thank E. Marberg for suggesting the example in Section \ref{logModule}.

\section{Preliminaries}\label{SectionPreliminaries}

Fix a finite set $N$ and a total order $\leq$ on $N$ (for example, $N=\{1,2,\ldots,n\}$ with $1<2<\cdots<n$).  We say that a subset $K\subseteq N$ is an \emph{interval} if for all $k,l\in K$ with $k<l$, we have $\{n\in N\mid k<n<l\}\subseteq K$.

This section introduces the notation we will use for set partitions, the notion of a supercharacter theory, and gives both the supercharacters and their modules for the supercharacter theory of $\UT_N$ used by this paper.

\subsection{Set partitions}

A \emph{set partition $\lambda$ of $N$} is a set of pairs $(i,j)\in N\times N$ with $i<j$ such that if $(i,k), (j,l)\in \lambda$, then $i=j$ if and only if $k=l$.  

Let
$$\scS_N=\{\text{set partitions of $N$}\}.$$

We typically view these set partitions diagrammatically as a family of arcs on a row of $|N|$ nodes so that if $(i,j)\in \lambda$, then there is an arc connecting the $i$th node to the $j$th node.  For example,
$$\lambda=\{(1,5), (2,4),(4,6)\} \longleftrightarrow
\begin{tikzpicture}[scale=.5,baseline=0cm]
	\foreach \x in {1,...,6}
	{	\node (\x) at (\x,0) [inner sep = -1pt] {$\bullet$};
		\node at (\x,.45) {$\scs \x$};}
	\draw (1)  to [out=-60,in=-120] (5);
	\draw (2) to [out=-60,in=-120] (4);
	\draw (4) to [out=-60,in=-120] (6);
\end{tikzpicture} \qquad \text{or}\qquad 
\begin{tikzpicture}[scale=.5,baseline=0cm]
	\foreach \x in {1,...,6}
	{	\node (\x) at (\x,0) [inner sep = -1pt] {$\bullet$};
		\node at (\x,-.45) {$\scs \x$};}
	\draw (1) to [out=60,in=120] (5);
	\draw (2) to [out=60,in=120] (4);
	\draw (4) to [out=60,in=120] (6);
\end{tikzpicture}.
$$
\begin{remark}
There are at least two natural choices for the arcs, either above or below the nodes.  We will use both, letting set partitions with arcs above correspond to upper-triangular matrices (superclasses) and arcs below correspond to lower-triangular matrices (supercharacters).
\end{remark}

 We typically refer to $(i,j)$ as an \emph{arc} of $\lambda$ and write $(i,j)=i\smile j$ or $(i,j)=i\frown j$, depending on our desired orientation.  For each $i\smile j\in  \lambda$, we call $i$ the \emph{left endpoint} of $i\smile j$ and $j$ the \emph{right endpoint} of $i\smile j$.  In general,
 \begin{align*}
 \epl{\lambda}&=\{i\in N\mid i\smile j\in \lambda, \text{for some } j\in N\}\\
 \epr{\lambda}&=\{j\in N\mid i\smile j\in \lambda, \text{for some } i\in N\}. 
 \end{align*}
 give the full sets of left and right endpoints of $\lambda$.

We obtain the more traditional version of set partitions by defining the \emph{blocks} $\bl(\lambda)$ of  $\lambda\in \scS_N$ to be the set of equivalence classes on $N$ given by the transitive closure of $i\sim j$ if $i\smile j\in \lambda$.   For example,
$$ \bl\Bigg(\begin{tikzpicture}[scale=.5,baseline=0cm]
	\foreach \x in {1,...,6}
	{	\node (\x) at (\x,0) [inner sep = -1pt] {$\bullet$};
		\node at (\x,.45) {$\scs \x$};}
	\draw (1) to [out=-60,in=-120] (5);
	\draw (2) to [out=-60,in=-120](4);
	\draw (4) to [out=-60,in=-120] (6);
\end{tikzpicture}\Bigg)=\{\{1,5\},\{2,4,6\},\{3\}\}.$$

There is an involution $\dag:\scS_N\rightarrow \scS_N$ given by flipping the diagram across the middle, or if $w_0\in S_N$ is the permutation of $N$ that reverses the order of the elements, then
\begin{equation}\label{SetPartitionInvolution}
\dag(\lambda)=\{w_0(j)\smile w_0(i) \mid i\smile j\in \lambda\},
\end{equation}
so for example,
$$\dag\bigg(\begin{tikzpicture}[scale=.5,baseline=0cm]
	\foreach \x in {1,...,6}
	{	\node (\x) at (\x,0) [inner sep = -1pt] {$\bullet$};
		\node at (\x,.45) {$\scs \x$};}
	\draw (1) to [out=-60,in=-120] (5);
	\draw (2) to [out=-60,in=-120](4);
	\draw (4) to [out=-60,in=-120] (6);
\end{tikzpicture}\bigg)=
\begin{tikzpicture}[scale=.5,baseline=0cm]
	\foreach \x in {1,...,6}
	{	\node (\x) at (\x,0) [inner sep = -1pt] {$\bullet$};
		\node at (\x,.45) {$\scs \x$};}
	\draw (2) to [out=-60,in=-120] (6);
	\draw (3) to [out=-60,in=-120] (5);
	\draw (1) to [out=-60,in=-120] (3);
\end{tikzpicture}.
$$

There are a number of natural statistics on set partitions, and they turn out to have a nice algebraic structure \cite{CDKR}; the most important ones for us come from nestings and crossings.  A \emph{crossing} in a set-partition $\lambda$ is a pair of arcs $i\smile k,j\smile l\in \lambda$ such that $i<j<k<l$.  A \emph{nesting} in a set partition $\lambda$ is a pair of arcs $i\smile l,j\smile k\in \lambda$ such that $i<j<k<l$.  For $\lambda,\mu\in \scS_N$ and $A\subseteq N$, let
\begin{align*}
\nst_\mu^\lambda &= \{(i\smile l,j\smile k)\in \lambda\times \mu\mid i<j<k<l\}\\
\nst_A^\lambda &= \{(i\smile l,j)\in \lambda\times A\mid i<j<l\}\\
\crs(\lambda) & = \{(i\smile k,j\smile l)\in \lambda\times \lambda\mid i<j<k<l\}.
\end{align*}
Define the set of noncrossing set-partitions to be 
$$\scS^{\asymp}_N=\{\lambda\in \scS_N\mid \crs(\lambda)=\emptyset\}.$$

Define a function
$$\begin{array}{r@{\ }c@{\ }c@{\ }c}\asymp: &\scS_N & \longrightarrow  &\scS^{\asymp}_N\\
& \lambda & \mapsto & \uncr{\lambda},\end{array}$$
where $\uncr{\lambda}$ is the unique set partition in $\scS^{\asymp}_N$ that has the same left endpoints and the same right endpoints as $\lambda$ (in $\uncr{\lambda}$, the first right end-point from the left must be connected to the closest left endpoint, etc.).  For example,
$$\underset{\asymp}{\begin{tikzpicture}[scale=.5,baseline=0cm]
	\foreach \x in {1,...,6}
	{	\node (\x) at (\x,0) [inner sep = -1pt] {$\bullet$};
		\node at (\x,.45) {$\scs \x$};}
	\foreach \x in {1,2,4}
		\node at (\x,.8) {$\scs \curvearrowbotleft$};
	\foreach \x in {4,5,6}
		\node at (\x,.8) {$\scs \curvearrowbotright$};
	\draw (1) to [out=-60,in=-120] (4);
	\draw (2) to [out=-60,in=-120] (5);
	\draw (4) to [out=-60,in=-120] (6);
\end{tikzpicture}}
=\begin{tikzpicture}[scale=.5,baseline=0cm]
	\foreach \x in {1,...,6}
	{	\node (\x) at (\x,0) [inner sep = -1pt] {$\bullet$};
		\node at (\x,.45) {$\scs \x$};}
	\foreach \x in {1,2,4}
		\node at (\x,.8) {$\scs \curvearrowbotleft$};
	\foreach \x in {4,5,6}
		\node at (\x,.8) {$\scs \curvearrowbotright$};
	\draw (2) to [out=-60,in=-120] (4);
	\draw (4) to [out=-60,in=-120](5);
	\draw (1) to [out=-60,in=-120] (6);
\end{tikzpicture},
$$
where we have marked the left endpoints with $\curvearrowbotleft$ and the right endpoints with $\curvearrowbotright$.

\begin{remark}
One can also obtain $\uncr{\lambda}$ from $\lambda$ by iteratively uncrossing each crossing $\{i\smile k,j\smile l\}$ into a nesting $\{i\smile l,j\smile k\}$.  Since this map changes neither the set of left endpoints nor the set of right endpoints, the order in which we ``resolve" the crossings does not matter.
\end{remark}

For $A,C\subseteq N$, let
$$\wt_C^\uparrow (A)= \#\{c\in C\mid a<c, a\in A\}=\wt_A^\downarrow(C).$$
Note that for $A\subseteq N$ and $\lambda\in \scS_N$,
\begin{equation}\label{NstWtConversion}
\nst_A^\lambda=\wt_{\epr{\lambda}}^\uparrow(A)-\wt_{\epl{\lambda}}^\uparrow(A).
\end{equation}
It follows that $\nst_A^\lambda$ only depends on the left and right endpoints of $\lambda$, so 
$$\nst_A^\lambda=\nst_A^{\uncr{\lambda}}.$$

%
\subsection{Supercharacters of $\UT_N$} \label{SectionSupercharacters}

Supercharacter theories were  originally developed to study the representation theory of $\UT_N$.  The first such theory was developed by \cite{An95}, and it was generalized to algebra groups in \cite{DI08}.  The theory we use below is slightly coarser (and more combinatorial), and it was first used in \cite{BT}.  The study of supercharacters has seen a fair amount of interest from a variety of point of views in recent years, including the Hopf structure \cite{ABT, Be,Ma},  good supercharacter theories for unipotent groups \cite{AN,An}, and generalizations to structures other than finite groups \cite{ANi,Ke}.

A \emph{supercharacter theory} $\scf(G)$ of a finite group $G$ is a subspace of the space of class functions $\cf(G)$  of $G$ such that $\scf(G)$ is a $\CC$-subalgebra under the two products
$$\begin{array}{rccc} \circ: & \cf(G)\otimes \cf(G) & \longrightarrow & \cf(G)\\
&\chi\otimes \psi &\mapsto & \begin{array}{@{}r@{\ }c@{}c@{}}\chi\circ\psi: G & \rightarrow & \CC\\ g & \mapsto & \dd \frac{1}{|G|} \sum_{h\in G} \chi(h)\overline{\psi(h^{-1}g)}\end{array}\end{array}$$
and 
$$\begin{array}{rccc} \odot: & \cf(G)\otimes \cf(G) & \longrightarrow & \cf(G)\\
&\chi\otimes \psi &\mapsto &\begin{array}{@{}r@{\ }c@{}c@{}}\chi\odot\psi: G & \rightarrow & \CC\\ g & \mapsto &\chi(g)\psi(g).\end{array}\end{array}$$

\begin{remark}
This definition gives a supercharacter theory as a special kind of Schur ring (\cite{He} has a nice description of the relationship).  There is an alternate definition by Diaconis--Isaacs \cite{DI08} which stresses the partitions of $G$ and $\Irr(G)$ that arise out of the two distinguished bases, described below.
\end{remark}

Every such a supercharacter theory has two distinguished bases (up to scaling) given by
\begin{align*}
\scf(G)&= \CC\spanning\{\kappa_A\mid A\in \cK\}\\
	 &= \CC\spanning\{\sum_{\theta\in Y} \theta(1)\theta \mid Y\in \cX\},
\end{align*} 
where $\cK$ is a partition of $G$, $\cX$ is a partition of the irreducible characters $\Irr(G)$, and for $g\in G$,
$$\kappa_A(g)=\left\{\begin{array}{@{}ll} 1 & \text{if $g\in A$,}\\ 0, & \text{if $g\notin A$}\end{array}\right.$$
We typically call the blocks of $\cK$ \emph{superclasses}. For each $Y\in \cX$ we choose constants $c_Y\in \QQ_{>0}$ such that 
$$\chi^Y=c_Y\sum_{\theta\in Y}\theta(1)\theta$$
is a character.  The resulting characters $\chi^Y$ are called \emph{supercharacters}.

\begin{remark}
In representation theory, there is always a tension between working with the full Wedderburn component of an irreducible module or the irreducible module itself.  The choice of $c_Y$ can be thought of a choice along this spectrum for each supercharacter $\chi^Y$.
\end{remark}

Let $\UT_N$ be the subgroup of unipotent upper-triangular matrices of the general linear group $\GL_N(\FF_q)$ over the finite field $\FF_q$ with $q$ elements, and let $\UT_N\subseteq B_N\subseteq\GL_N(\FF_q)$ be the Borel subgroup of uppertriangular matrices.    While there are many supercharacter theories of $\UT_N$, we will focus on the supercharacter theory $\scf(\UT_N)$ that is a slight coarsening of the canonical supercharacter theory for algebra groups given by \cite{DI08}.  Let 
\begin{equation}\label{logAlgebra}
\fkut_N=\UT_N-\Id_N
\end{equation}
be the nilpotent $\FF_q$-algebra of strictly uppertriangular matrices.  Then the superclasses of $\scf(\UT_N)$ are given by the two-sided orbits
$$\Id_N+B_N\backslash \fkut_N/B_N.$$

The dimension of $\scf(\UT_N)$ turns out to be Bell number $|\scS_N|$.  In fact, for every superclass of $\UT_N$ there exists $\mu\in \scS_N$ and a distinguished element $u_\mu$ in the superclass such that 
$$(u_\mu)_{ij}=\left\{\begin{array}{ll} 1, & \text{if $i\larc{}j\in \mu$ or $i=j$,}\\
0, & \text{otherwise}.\end{array}\right.$$

We will construct the supercharacters explicitly in the next section.  However, the following character formula gives them explicitly and will be useful, below.  It was first proved for a slightly finer supercharacter theory (indexed by labeled set partitions) for $\mathrm{char}(\FF_q)$ sufficiently large in \cite{An96} and then for general $q$ in \cite{ADS}.  

\begin{proposition}\label{SupercharacterFormula}
For $\lambda,\mu \in \scS_N$,
$$\chi^\lambda(u_\mu)=\left\{\begin{array}{ll}
(q-1)^{|\lambda-\mu|} (-1)^{|\lambda\cap \mu|} q^{\nst_{N}^{\lambda}-\nst_{\mu}^\lambda} & \begin{array}{@{}l@{}}\text{if $i<j<k$, $i\smile k\in \lambda$}\\ \text{implies $i\frown j, j\frown k\notin \mu$,}\end{array}\\ 0 & \text{otherwise.}
\end{array}\right.$$
\end{proposition}
In particular, note that the trivial character $\One$ is the supercharacter $\chi^\emptyset$ associated with the empty set partition of $N$, and 
$$\chi^\lambda(1)=(q-1)^{|\lambda|} q^{\nst_N^\lambda}.$$
It also follows from the formula that
$$\chi^\lambda=\bigodot_{i\smile j\in \lambda} \chi^{i\smile j}.$$
Let 
$$\scM_N=\{\text{multi-sets of arcs with endpoints in $N$}\}.$$
For $\lambda\in \scM_N$ we define the character
$$\chi^{\lambda}=\bigodot_{i\smile j\in \lambda} \chi^{i\smile j}.$$

In general, if $K\subseteq N$, then 
$$\UT_K\cong \{u\in \UT_N\mid (u-\Id_N)_{ij}\neq 0 \text{ implies  $i,j\in K$}\}\subseteq \UT_N.$$
An explicit restriction formula for restricting supercharacters from $\UT_N$ to $\UT_K$ seems to be difficult.  The paper \cite{Th10} gives iterative algorithms for both computing restrictions and pointwise products of supercharacters; heuristically, arcs with illegal endpoints shrink in all possible ways.  The following lemma gives the first steps of this algorithm.   

\begin{lemma} \label{SimpleRestrictionTensor}
Let $N$ be a set
\begin{enumerate}
\item[(a)] If $N=K\cup \{n\}$ with $K<\{n\}$, then for $k\in K$,
$$\Res_{\UT_K}^{\UT_N}(\chi^{k\smile n})=(q-1)\Big(\chi^\emptyset+\sum_{l\in K\atop k<l}\chi^{k\smile l}\Big).$$
\item[(b)]  If $i<j<l$, then
$$\chi^{i\smile l}\odot \chi^{j\smile l}=(q-1)\Big(\chi^{i\smile l} +\sum_{j<k<l} \chi^{\{i\smile l,j\smile k\}}\Big).$$
\end{enumerate}
\end{lemma}
\begin{remark}
In Section \ref{SectionSupermodules} below we discuss a natural involution $\dag$ on the supercharacters of $\UT_N$, corresponding to the combinatorial $\dag$ on set partitions.  Using this involution one can also obtain a ``flipped" version of the above lemma.
\end{remark}

%
%
\subsection{The supercharacter $\UT_N$-modules} \label{SectionSupermodules}

Fix a nontrivial homomorphism
$$\vartheta:\FF_q^+\longrightarrow \CC^\times.$$
The $\FF_q$-vector space of $|N|\times |N|$ matrices $\fkgl_N$ with entries in $\FF_q$ 
decomposes in terms of  uppertriangular matrices $\fkb_N$ and strictly lower triangular matrices $\fklt_N$, so 
$$\fkgl_N=\fklt_N\oplus \fkb_N.$$
Define the $\CC$-vector space
$$V_N=\CC\spanning\{\fklt_N\}.$$
We can define several left $\UT_N$-module structures on $V_N$ using the fact that $\fklt_N$ is a canonical set of coset representatives in $\fkgl_N/\fkb_N$.  For $v\in \fkgl_N$, define 
$$\{\bar{v}\}=(v+\fkb_N)\cap \fklt_N.$$
We use the left multiplication action on $\fkgl_N/\fkb_N$ to confer a $\UT_N$-module $\overleftarrow{V}_N$ structure on the space $V_N$ by defining
\begin{equation} \label{LeftAction}
 u\lact v= \vartheta\big(\tr((u-1)v)\big) (\overline{uv}) \quad \text{for $u\in \UT_N$, $v\in \fklt_N$.}
 \end{equation}
Similarly, the right multiplication action on $\fkgl_N/\fkb_N$ gives a $\UT_N$-module $\overrightarrow{V}_N$ structure on the space $V_N$  by
\begin{equation} \label{RightAction}
u\ract v=
 \vartheta\big(\tr(v(u^{-1}-1))\big) (\overline{vu^{-1}}) \quad \text{for $u\in \UT_N$, $v\in \fklt_N$.}
\end{equation}
\begin{remark}
In constructing the supercharacters of $\UT_N$ it is more common to construct a module structure on $\fkut_N^*$, where $\fkut_N\subseteq \fkgl_N$ is as in (\ref{logAlgebra}) \cite{DI08}.  The two actions above correspond to the two canonical actions on $\fkut_N^*$.   This paper translates this picture to matrices to make studying submodules more straight-forward.
\end{remark}

Let $\dag:\fkgl_N\rightarrow \fkgl_N$ be the anti-involution obtained by transposing across the anti-diagonal, or for $g\in \fkgl_N$,
$$g^\dag=w_0\mathrm{Transpose}(g)w_0,\qquad \text{where}\qquad w_0=\left[\begin{array}{ccc} 0 & & 1\\ & \adots & \\ 1 & & 0\end{array}\right].$$
Note that $\dag$ restricts to an anti-involution $\dag:\fklt_N\rightarrow \fklt_N$ which we may extend linearly to $V_N$.  While $\overleftarrow{V}_N\cong \overrightarrow{V}_N$ are always isomorphic to the regular module of $\UT_N$, subspaces are not necessarily invariant under both actions.  However, we get a bijection
\begin{equation}\label{ModuleInvolution}
\begin{array}{r@{\ }c@{\ }c@{\ }c}\dag: & \left\{\begin{array}{@{\ }c@{\ }} \text{Submodules}\\ \text{of $\overleftarrow{V}_N$}\end{array}\right\} & \longleftrightarrow &  \left\{\begin{array}{@{\ }c@{\ }} \text{Submodules}\\ \text{of $\overrightarrow{V}_N$}\end{array}\right\}\\ & M & \mapsto & \dag(M). \end{array}
\end{equation}

For $\lambda\in \scS_N$, we define a canonical basis element $v_\lambda\in \fklt_N$ given by
$$(v_\lambda)_{kj}=\left\{\begin{array}{ll} 1 & \text{if $j\smile k\in \lambda$},\\ 0 & \text{otherwise.}\end{array}\right.$$
Then
\begin{equation*}
\fklt_N=\{\overline{av_\lambda b}\mid a, b\in B_N,\lambda\in \scS_N\}.
\end{equation*}
In this notation, for a fixed $b_0\in B_N$, the character of the submodule
$$V^\lambda\cong \CC\spanning\{\overline{av_\lambda b_0}\mid a\in B_N\}$$
is the supercharacter 
$$\chi^\lambda=\frac{|B_Nv_\lambda|}{|B_Nv_\lambda B_N|}\sum_{v\in B_Nv_\lambda B_N} \vartheta(\tr( \cdot\ v))$$
of $\UT_N$ (though the formula in Proposition \ref{SupercharacterFormula} is more useful for our purposes).

The notation of $\dag:\scS_N\rightarrow \scS_N$ in (\ref{SetPartitionInvolution}) now matches up with the (\ref{ModuleInvolution}).
\begin{proposition} \label{LeftRightModules}
We have
$$M\cong \bigoplus_{\lambda\in \scS_N} m_\lambda V^\lambda\subseteq \overleftarrow{V}_N
\quad
\text{if and only if} 
\quad\dag(M)\cong\bigoplus_{\lambda\in \scS_N} m_\lambda V^{\dag(\lambda)} \subseteq  \overrightarrow{V}_N.$$
\end{proposition}
\begin{proof}
Since $\dag(v_\lambda)=v_{\dag(\lambda)}$, the definition of $V^\lambda$ gives the result.
\end{proof}

\begin{remark}
Since most of our work will be studying submodules of $\overleftarrow{V}_N$, we will typically omit the arrow from the notation.  However, in Section \ref{SectionOnionPeel} we will require both kinds of submodules, so in that case we will use the arrows to help differentiate the actions.
\end{remark}

\section{Rainbow supercharacter restrictions}\label{SectionRainbow}

This section studies the special case where we have $N\subseteq N'$, where $N'=N_-\sqcup N\sqcup N_+$ with $N_-<N<N_+$ and $|N_-|=|N_+|=\ell$. 
Let $w:N_-\rightarrow N_+$ be the unique bijection such that $i<j$ if and only if $w(j)<w(i)$ for all $i,j\in N_-$.  Then the set partition
$$\{i\smile w(i)\mid i\in N_-\}=
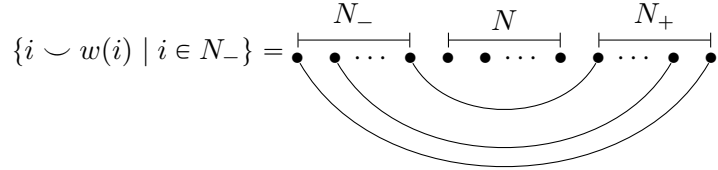
\begin{tikzpicture}[scale=.5,baseline=0]
	\foreach \x in {1,2,4,5,6,8,9,11,12}
		\node (\x) at (\x,0) [inner sep =-1pt] {$\bullet$};
	\foreach \x in {3,7,10}
		\node at (\x,0) {$\cdots$};
	\draw (1) to [out=-60,in=-120] (12);	
	\draw (2) to [out=-60,in=-120] (11);
	\draw (4) to [out=-60,in=-120] (9);	
	\draw[|-|] (1,.5) -- node [above] {$N_-$} (4,.5);	
	\draw[|-|] (5,.5) -- node [above] {$N$} (8,.5);
	\draw[|-|] (9,.5) -- node [above] {$N_+$} (12,.5);
\end{tikzpicture}
$$
looks a little like an upside-down rainbow, so we've informally dubbed the corresponding supercharacter a \emph{rainbow supercharacter}.  For our purposes, however, it is essentially equivalent to contract $N_-$ and $N_+$ into single points $\{n_-\}$ and $\{n_+\}$ and have a multi-set of arcs passing between these two points.  In fact, by comparing their values on $\UT_N$ using Proposition \ref{SupercharacterFormula}, we see that
$$\Res_{\UT_N}^{\UT_{N'}}(\chi^{\{i\smile w(i)\mid i\in N_-\}})=q^{2\binom{\ell}{2}} \Res_{\UT_N}^{\UT_{N\cup \{n_-,n_+\}}} (\underbrace{\chi^{n_-\smile n_+}\odot \cdots \odot \chi^{n_-\smile n_+}}_{\text{$\ell$ terms}}).$$
We will write
\begin{equation}\label{RainbowNotation}
n_-\underset{\ell}{\smile} n_+ = \{\underbrace{n_-\smile n_+,\ldots,n_-\smile n_+}_{\text{$\ell$ terms}}\}\in \scM_{N\cup\{n_-,n_+\}}.
\end{equation}
The fundamental idea of this paper is to construct intermediate modules to make the restriction of supercharacters more manageable.  This section is meant to serve as a model for this approach.  We begin in Section \ref{SectionBinomials} by developing a poset analogue of $q$-binomial coefficients that will be helpful in understanding the modules we construct in Sections \ref{SectionColumnSet} and \ref{SectionColumnCardinality}.  Section \ref{SectionRainbowRestriction} then uses the intermediate modules to decompose the rainbow supercharacters.

%
%
\subsection{Poset $q$-binomial coefficients} \label{SectionBinomials}

For $n,k\in \ZZ_{\geq 0}$, let
$$[n]=\frac{q^n-1}{q-1}\qquad \text{and}\qquad \Pqbin{n}{k}{q}=\frac{[n]!}{[k]![n-k]!}=\#\left\{\begin{array}{@{\ }c@{\ }}\text{dimension $k$}\\ \text{subspaces of $\FF_q^n$}\end{array}\right\}$$
be the usual $q$-integer and $q$-binomial.  

Let $\cP$ be a (finite) poset.  Given a subset $A\subseteq \cP$, let
\begin{align*}
\wt^\cP(A) =\sum_{a\in A}\wt^\cP(a),\qquad \text{where}\qquad \wt^\cP(a) = \#\{b\in \cP\mid b\succ_\cP a\}.
\end{align*}

For $n,k\in \ZZ_{\geq 0}$, the $\cP$-binomial coefficient is   
\begin{equation}
\Pqbin{\cP}{k}{q}=\sum_{A\subseteq \cP\atop |A|=k} q^{\wt^\cP(A)}.
\end{equation}
Note that $\wt^\cP(a)$ can also be expressed in terms of the size of the upper ideal containing $a$, and if $\cT$ is the usual order on $N$, then $\wt^\uparrow_N(A)=\wt^{\cT}(A)$.

\begin{examples}\hfill

\begin{itemize}
\item If $\cP$ is the poset with no relations on $n$ elements, then 
$$\Pqbin{\cP}{k}{q}=\sum_{A\subseteq \cP\atop |A|=k} q^0=\binom{|\cP|}{k}.$$
\item If $\cP$ is a total order (say $1<2<\cdots<n$), then 
\begin{align*}
\Pqbin{\cP}{k}{q} &= q^{1+2+\cdots+(k-1)} \hspace{-.75cm}\sum_{1\leq a_1<a_2<\cdots< a_k\leq n} \hspace{-.75cm}q^{n-a_1-(k-1)+n-a_2-(k-2)+\cdots +n-a_k-0}=q^{\binom{k}{2}} \Pqbin{|\cP|}{k}{q}.
\end{align*} 
Note that this is also $e_k(1,q,\ldots, q^{n-1})$, where $e_k(X_1,\ldots, X_n)$  is the $k$th elementary symmetric polynomial \cite[Exercise I.2.3]{Mac}. 
\end{itemize}
\end{examples}

\begin{remark}
In general $\Pqbin{\cP}{k}{q}\neq \Pqbin{\cP}{n-k}{q}$, though the coefficient sequences of the polynomial of one is the reverse coefficient sequence of the other.  That is, if
$$b_\cP(n,k,r,s)=\sum_{A\sqcup B=\{1,2,\ldots, n\}\atop |A|=k,|B|=n-k} r^{\wt(A)}s^{\wt(B)},$$
then 
$$\Pqbin{\cP}{k}{q}=b_\cP(n,k,q,1)\qquad \text{and}\qquad \Pqbin{\cP}{n-k}{q}=b_\cP(n,k,1,q).$$ 
\end{remark}

The usual method of defining a multinomial coefficient turns out to be less useful in our case, since
$$\Pqbin{\cP}{k_1,k_2,\ldots, k_\ell}{q}=\sum_{A_1\sqcup A_2\sqcup\cdots\sqcup A_\ell=\cP\atop |A_j|=k_j} q^{\sum_{j=1}^\ell \wt^\cP(A_j)}=\binom{|\cP|}{k_1,k_2,\ldots,k_\ell} q^{\wt^\cP(\cP)}.$$
However, there is a different way to pick multiple disjoint subsets of a poset that  is more useful.  Fix subsets $\cP_1\sqcup \cP_2\sqcup \cdots \sqcup \cP_{\ell}=\cP$, and defining
\begin{equation}\label{MultinomialPosetPartition}
\Pqbin{\cP}{k_1\subseteq \cP_1,k_2\subseteq \cP_2,\ldots, k_{\ell-1}\subseteq \cP_{\ell-1},k_\ell}{q}=\bigg(\prod_{j=1}^{\ell} \sum_{A_j\subseteq \cP_j\atop |A_j|=k_j}q^{\wt^\cP(A_j)} \bigg).
\end{equation}
Now, if $\cA\sqcup \cB=\cP$, then we obtain a kind of symmetry
$$\Pqbin{\cP}{k\subseteq \cA,|\cP|-k}{q}=\Pqbin{\cP}{|\cP|-k\subseteq \cB, k}{q}.$$

These coefficients satisfy a family of recursive relations.
\begin{proposition}
Let $a\in \cP$, $\cP'=\cP-\{a\}$ and let $\cP_a$ be the restriction of $\cP$ to the set $\cP_a=\{a'\mid a'\prec a\}$.  Then 
$$\Pqbin{\cP}{k}{q}=\sum_{j=0}^{k}  q^j\bigg( q^{\wt^\cP(a)}\Pqbin{\cP'}{j\subseteq \cP_a,k-j-1}{q} +\Pqbin{\cP'}{j\subseteq \cP_a,k-j}{q}\bigg).$$
\end{proposition}
\begin{proof}
We sort subsets of $\cP$ into $k$-subsets that contain $a$ and $k$-subsets that do not contain $a$.  If a $k$-subset $a\in A\subseteq \cP$  then $A-\{a\}\subseteq \cP'$ is a subset of size $k-1$.  In this case,
$$\wt^\cP(A)=|A\cap \cP_a|+\wt^{\cP}(A\cap \cP_a)+\wt^\cP(a)+\wt^\cP(A-\cP_a).$$
If a $k$-subset $a\notin A\subseteq \cP$, then $A\subseteq \cP'$, and 
$$\wt^\cP(A)=|A\cap \cP_a|+\wt^{\cP}(A\cap \cP_a)+ \wt^\cP(A-\cP_a).$$
The result now follows from summing over possible cardinalities of $A\cap \cP_a$.
\end{proof}

Perhaps the most useful recursion corresponds to removing some minimal  element in the poset.

\begin{corollary}\label{MinimimalElementRecursion}  
 Let $a$ be a minimal element of a poset $\cP$. Let $\cP'=\cP-\{a\}$.  Then
$$\Pqbin{n}{k}{\cP}= q^{\wt(a)}\Pqbin{\cP'}{k-1}{q} +\Pqbin{\cP'}{k}{q}.$$
\end{corollary}

\begin{example} For the posets with nearly no relations we have
\begin{align*}
\Pqbin{\begin{tikzpicture}[baseline=0, scale=.5]
	\foreach \x/\y in {0/0,1/0, 3/0, 2/1}
		\node (\x\y) at (\x,\y) [inner sep=0pt] {$\bullet$}; 
	\foreach \z in {0,1,3}
		\draw (\z0) -- (21);
	\foreach \x/\n in {0/1,1/2,3/n-1}
		\node at (\x,-.5) {$\scs\n$};
	\node at (2,1.5) {$\scs n$};
	\node at (2,0) {$\cdots$};
\end{tikzpicture}}{k}{q} &= \sum_{A\subseteq \cP\atop |A|=k, n\in A} q^{k-1} + \sum_{A\subseteq \cP\atop |A|=k, n\notin A} q^{k} = q^{k-1}\bigg(\binom{n-1}{k-1} + q\binom{n-1}{k}\bigg). 
\end{align*}
and  
\begin{align*}
\Pqbin{\begin{tikzpicture}[baseline=0,scale=.5]
	\foreach \x/\y in {0/1,1/1, 3/1, 2/0}
		\node (\x\y) at (\x,\y) [inner sep=0pt] {$\bullet$}; 
	\foreach \z in {0,1,3}
		\draw (\z1) -- (20);
	\foreach \x/\n in {0/2,1/3,3/n}
		\node at (\x,1.5) {$\scs\n$};
	\node at (2,-.5) {$\scs 1$};
	\node at (2,1) {$\cdots$};
\end{tikzpicture}}{k}{q} &= \sum_{A\subseteq \cP\atop |A|=k, 1\in A} q^{n-1} + \sum_{A\subseteq \cP\atop |A|=k, 1\notin A} q^{0} = q^{n-1}\binom{n-1}{k-1} + \binom{n-1}{k}. 
\end{align*}
\end{example}

\subsubsection{Main example}

Let $\lambda\in\scS^{\asymp}_N$. Since there are no crossings in $\lambda$, we can define a poset depending on $\lambda$ based on whether blocks are nested or not.  Let $\cP(\lambda)$ be the poset on $\bl(\lambda)$ given by $a\prec b$ if either
\begin{itemize}
\item $|a|>1$ and there exist $j,k\in a$ and $i,l\in b$ such that $i<j<k<l$, or 
\item $a=\{j\}$ and there exist $i,k\in b$ such that $i<j<k$.
\end{itemize}
For example,
$$\lambda=
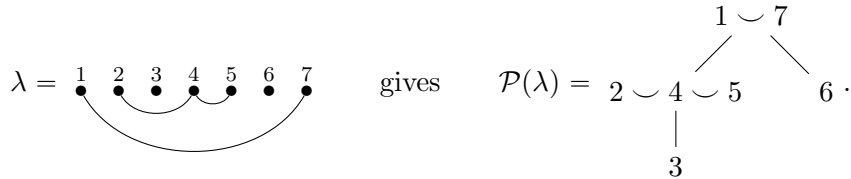
\begin{tikzpicture}[scale=.5,baseline=0cm]
	\foreach \x in {1,...,7}
	{	\node (\x) at (\x,0) [inner sep = -1pt] {$\bullet$};
		\node at (\x,.45) {$\scs \x$};}
	\draw (1) to [out=-60,in=-120]   (7);
	\draw (2) to [out=-60,in=-120] (4);
	\draw (4) to [out=-60,in=-120]  (5);
\end{tikzpicture}\qquad \text{gives}\qquad 
\cP(\lambda)=
\begin{tikzpicture}[scale=.5,baseline=0cm]
	\node (4) at (2,2) {$1\smile 7$};
	\node (3) at (4, 0) {6} ;
	\node (2) at (0,0) {$2\smile 4\smile 5$};
	\node (1) at (0,-2) {3};
	\draw (1) -- (2);
	\draw (2) -- (4);
	\draw (3) -- (4);
\end{tikzpicture}.$$
Note that for each $\lambda\in \scS^{\asymp}_N$ the poset $\cP(\lambda)$ is a forest where each connected component has a unique maximal element.  In fact, all such forests arise in this way (as $N$ and $\lambda$ vary).

Thus, we obtain a function
$$\scS_N\overset{\asymp}{\longrightarrow} \scS^{\asymp}_N \overset{\cP}{\longrightarrow} \left\{\begin{array}{@{}c@{}} \text{forests where each}\\ \text{connected component has}\\ \text{a unique maximal element}\end{array}\right\}.$$

%
%
\subsection{Column set submodules} \label{SectionColumnSet}

For $K\subseteq N$,  the submodule
$$V_N^K=\CC\spanning\{\fklt_N^K\}\subseteq V_N, \qquad \text{where}\qquad \fklt_N^K=\{v\in \fklt_N\mid v_{ji}\neq 0 \text{ implies } i\in K\},$$
 decomposes completely into super-modules.

\begin{proposition}\label{V^KDecomposition} For $K\subseteq N$,
$$V^K_N\cong \bigoplus_{\lambda\in \scS_{N}\atop \epl{\lambda}\subseteq K} q^{\nst_\lambda^\lambda+\nst_{K-\epl{\lambda}}^\lambda} V^\lambda. $$
\end{proposition}
\begin{proof}
The proof takes every basis element in $V_N^K$ and finds it as a basis element for some copy of $V^\lambda$.  Note that 
$$V_N^K=\CC\spanning\{\overline{av_\lambda b}\mid a\in B_N, b\in B_K, \epl{\lambda}\subseteq K\}.$$
For a fixed $b_0\in B_K$ and $\lambda$,
$$\CC\spanning\{\overline{av_\lambda b_0}\mid a\in B_N\}\cong V^\lambda,$$
so for each $\lambda$ it suffices to determine the size of 
$$\frac{|\overline{B_Nv_\lambda B_K}|}{|\overline{B_Nv_\lambda}|}=\frac{|\overline{B_Nv_\lambda}||\overline{v_\lambda B_K}|}{|\overline{B_Nv_\lambda}||\overline{B_Nv_\lambda}\cap \overline{v_\lambda B_K}|}=\frac{|\overline{v_\lambda B_K}|}{|\overline{B_Nv_\lambda}\cap \overline{v_\lambda B_K}|}.$$  
The elements of $K$ fall into two categories:
\begin{description}
\item[Case 1.] $k\in K- \epl{\lambda}$,
\item[Case 2.] $k\in K\cap \epl{\lambda}$. If $k\smile j\in \lambda$ and $l\smile i\in \lambda$ with $l>k>i$, then either $k\smile j$ crosses or is nested in $l\smile i$.
\end{description}
Therefore, we conclude that
$$\frac{|\overline{v_\lambda B_K}|}{|\overline{B_Nv_\lambda}\cap \overline{v_\lambda B_K}|}=\frac{q^{\nst_K^\lambda}}{q^{\#\{i<j<k<l\mid i\smile k,j\smile l\in \lambda, j\in K\}}}=q^{\nst_\lambda^\lambda+\nst_{K-\epl{\lambda}}^\lambda},$$
as desired. 
\end{proof}

By the previous proposition the trace of $V_N^K$ will be a superclass function 
$$\psi_N^K=\tr(\cdot, V_N^K).$$
of $\UT_N$.  We can easily compute its trace directly on superclass representatives. 

\begin{proposition}  \label{V^KValue} For $K\subseteq N$ and $\mu\in \scS_N$,
$$\psi_N^K(u_\mu)=\left\{\begin{array}{ll}q^{\wt^\uparrow_{N-\epl{\mu}}(K)} & \text{if $|\epl{\mu}\cap K|=0$,}\\ 0 & \text{otherwise.}\end{array}\right.$$
\end{proposition}

\begin{proof}
Let $v\in V_N^K$.  Then
$$u_\mu \lact v =\vartheta(\tr((u_\mu-1) v)) \bigg[\sum_{l=1}^{|N|} (u_\mu)_{jl}v_{li}\bigg]_{|N|\leq j>i\geq 1}$$
Thus, $u_\mu\lact v=\vartheta(\tr((u_\mu-1) v))v$ if and only if for all $j\larc{}l\in \mu$, we have $v_{li}=0$ for all $i < j$. The trace then is
\begin{align*}
\psi_N^K(u_\mu)&=\sum_{{v\in \fklt_N^K\atop j\slarc{}l\in \mu \text{ implies}}\atop v_{li}=0, i<j} \vartheta(\sum_{j\slarc{}l\in \mu\atop j\in K} v_{lj})\\
&=\sum_{{v\in \fklt_N^K\atop j\slarc{}l\in \mu \text{ implies}}\atop v_{li}=0, i\leq j} \prod_{j\slarc{}l\in \mu\atop j\in K} \sum_{v_{lj}\in \FF_q} \vartheta(v_{lj})\\
&=0,
\end{align*}
unless $|\epl{\mu}\cap K|=0$.  If $|\epl{\mu}\cap K|=0$, then
$$\psi_N^K(u_\mu)=\sum_{{v\in \fklt_N^K\atop j\slarc{}l\in \mu \text{ implies}}\atop v_{li}=0, i\leq j} 1 = |\fklt_N^K|q^{-\wt_{\epl{\mu}}^\uparrow(K)}=q^{\wt^\uparrow_{N}(K)-\wt_{\epl{\mu}}^\uparrow(K)}=q^{\wt^\uparrow_{N-\epl{\mu}}(K)},$$
as desired. 
\end{proof}

\begin{example}
Proposition \ref{V^KValue} implies that 
\begin{itemize}
\item $V_N^\emptyset$ is the trivial module of $\UT_N$,
\item $V_N^N$ is isomorphic to the regular module of $\UT_N$.
\end{itemize}
\end{example}

We can also use Proposition \ref{V^KValue} to connect these modules with the restriction problem, as follows.

\begin{corollary} \label{V^KRealization}  Let $N'=N\cup\{n_+\}$ with $N<\{n_+\}$. For $K\subseteq N$,
$$\Res_{\UT_N}^{\UT_{N'}}\Big(\bigotimes_{k\in K} V^{k\smile n_+}\Big)\cong (q-1)^{|K|}V_N^K.$$
\end{corollary}
\begin{proof}
Take traces using Proposition \ref{SupercharacterFormula} on the left and Proposition \ref{V^KValue} on the right.
\end{proof}
\begin{remark}
 Corollary \ref{V^KRealization} implies that the left module $V^K_N$ is isomorphic to the right module
 $$\CC\spanning\{\UT_N^{K}\} \qquad \text{where}\qquad \UT_N^K=\{u\in \UT_N\mid (u-\Id_N)_{ij}\neq 0 \text{ implies } i\in K\},$$
under right multiplication.
\end{remark}

%
%
\subsection{Core modules} \label{SectionColumnCardinality}

We can clump the $V_N^K$ based on the size of $|K|$ to obtain the following new modules.  For $0\leq k\leq |N|$, let
\begin{equation}\label{V^kToV^K}
V_N^k=\bigoplus_{K\subseteq N\atop |K|=k} V_N^K.
\end{equation}
Note that $V_N^0=V_N^\emptyset$ and $V_N^{|N|}=V_N^N$, so these modules also interpolate between the trivial module and the regular module.

\begin{remark}
In the context of Section \ref{SectionOnionPeel}, it will make sense to call these modules \emph{core} modules.
\end{remark}

\begin{proposition}\label{V^kDecomposition} For $0\leq k\leq |N|$,
$$V_N^k\cong\bigoplus_{\lambda\in \scS_{N}\atop |\lambda|\leq k} q^{\nst_\lambda^\lambda} \Pqbin{\cP(\uncr{\lambda})}{k-|\lambda|}{q} V^\lambda.$$
\end{proposition}
\begin{proof}
By Proposition \ref{V^KDecomposition}, the multiplicity of $V^\lambda$ is
$$ \sum_{\epl{\lambda}\subseteq K\subseteq N\atop |K|=k} q^{\nst_\lambda^\lambda+\nst_{K-\epl{\lambda}}^\lambda}=
q^{\nst_\lambda^\lambda} \sum_{K\subseteq N-\epl{\lambda},\atop |K|=k-|\lambda|} q^{\nst_{K}^\lambda}.$$
Since each block in $\cP(\uncr{\lambda})$ has a unique point that is not in $\epl{\lambda}$, each element in $j\in N-\epl{\lambda}$ uniquely determines a block $\mathrm{bl}_j\in\cP(\uncr{\lambda})$.  Furthermore, $\epl{\lambda}=\epl{\uncr{\lambda}}$ and $\epr{\lambda}=\epr{\uncr{\lambda}}$, so for each $j\in N-\epl{\lambda}$, 
\begin{align*}
\nst_j^\lambda &=\#\{l\in \epr{\lambda}\mid j<l\}-\#\{k\in \epl{\lambda}\mid j<k\}\\
&=\#\{l\in \epr{\uncr{\lambda}}\mid j<l\}-\#\{k\in \epl{\uncr{\lambda}}\mid j<k\}\\
&= \wt_{\cP(\uncr{\lambda})}^\uparrow (\mathrm{bl}_j).
\end{align*}
Thus,
$$q^{\nst_\lambda^\lambda} \sum_{K\subseteq N-\epl{\lambda},\atop |K|=k-|\lambda|} q^{\nst_{K}^\lambda}=q^{\nst_\lambda^\lambda} \Pqbin{\cP(\uncr{\lambda})}{k-|\lambda|}{q},$$
as desired.
\end{proof}

For $0\leq k\leq |N|$, let
$$\psi_N^k=\tr(\cdot, V_N^k).$$

\begin{corollary} \label{V^kValue} For $0\leq k\leq |N|$ and $\mu\in \scS_N$,
$$\psi_N^k(u_\mu)= q^{\binom{k}{2}} \Pqbin{|N|-|\mu|}{k}{q}. $$
\end{corollary}
\begin{proof}
If $\cT_{N}$ is the fixed total order on the set $N$, then by Proposition \ref{V^KValue}, 
$$\psi_N^k(u_\mu)=\Pqbin{\cT_{N-\epl{\mu}}}{k}{q}=q^{\binom{k}{2}} \Pqbin{|N|-|\epl{\mu}|}{k}{q}.$$
Observe that $|\epl{\mu}|=|\mu|$.
\end{proof}

By Corollary \ref{V^kValue}, $\psi_N^k(u_\mu)$ only depends on the number $|\mu|$, and in fact these functions span the space 
$$\mathrm{bscf}(\UT_N)=\{\theta\in \scf(\UT_N)\mid \theta(u_\mu)=\theta(u_\nu)\text{ if $|\mu|=|\nu|$}\}.$$
It is also clear that this subspace is closed under pointwise products; however, while it is tempting to think it might give a coarser supercharacter theory, it does not, since the $\mathrm{bscf}(\UT_N)$ does not have a basis of orthogonal characters.

However, the decomposition of pointwise products is of some interest, and we will use the base case of the following theorem in the next section.

\begin{theorem}\label{V^kTensor}
For $0\leq j\leq k\leq |N|$,
$$V_N^j\otimes V_N^k\cong  \bigoplus_{m=0}^{j} q^{\binom{j-m}{2}}\Pqbin{k+m}{k+m-j,m,j-m}{q} V_N^{k+m} .$$
\end{theorem}

\begin{remark}
The $j=1$ case can be derived using Corollary \ref{V^KRealization} and Lemma \ref{SimpleRestrictionTensor} above.  However, it is less involved to just check that the answer is correct (as we do in the proof below).  
\end{remark}

\begin{proof}
Induct on $j$.
 
If $j=0$, then since we are tensoring with the trivial module the result is clear.   If $j=1$, then by taking traces we use Proposition \ref{V^kValue} to compare for $\mu\in \scS_N$, 
\begin{align*}
[k]\psi_N^k(u_\mu)+[k+1]\psi_N^{k+1}(u_\mu) &=q^{\binom{k}{2}} \frac{ [k] [|N|-|\mu|]!}{[k]![|N|-|\mu|-k]!}+   q^{\binom{k+1}{2}} \frac{[k+1][|N|-|\mu|]!}{[k+1]![|N|-|\mu|-k-1]!}\\
&= q^{\binom{k}{2}} \frac{[|N|-|\mu|]!}{[k]![|N|-|\mu|-k-1]!}\Big(\frac{[k]}{[|N|-|\mu|-k]} +  q^{\binom{k}{1}}\Big)\\
&= q^{\binom{k}{2}} \frac{[|N|-|\mu|]!}{[k]![|N|-|\mu|-k]!}([k]+  q^{\binom{k}{1}}[|N|-|\mu|-k])\\
&=\frac{q^{\binom{k}{2}}[|N|-|\mu|]![|N|-|\mu|] }{[|N|-|\mu|-k]![k]!}\\
&=\psi_N^1(u_\mu)\odot\psi_N^k(u_\mu).
\end{align*} 

If $j>1$, then rewrite the base case,
$$\psi_N^j\otimes \psi_N^k=\frac{1}{[j]} \Big( \psi_N^1\odot\psi_N^{j-1}\odot\psi_N^k-[j-1]\psi_N^{j-1}\odot \psi_N^k\Big),$$
so that we can use induction to get
\begin{align*}
\psi_N^j\otimes \psi_N^k & = \frac{1}{[j]}\bigg( \sum_{m=0}^{j-1} q^{\binom{j-1-m}{2}} \Pqbin{k+m}{k+m-j+1,m,j-1-m}{q}\Big([k+m]\psi_N^{k+m}+[k+m+1]\psi_N^{k+m+1}\Big)\\
&\hspace*{.5cm} - \sum_{m=0}^{j-1} q^{\binom{j-1-m}{2}} \Pqbin{k+m}{k+m-j+1,m,j-1-m}{q}[j-1]\psi_N^{k+m}\bigg)\\
&=  \frac{1}{[j]}\bigg( q^{\binom{j-1}{2}} \Pqbin{k}{k-j+1,0,j-1}{q}\Big([k]-[j-1]\Big)\psi_N^k\\
&\hspace*{.5cm} +\sum_{m=1}^{j-1}\Big( q^{\binom{j-m}{2}} \Pqbin{k+m-1}{k+m-j,m-1,j-m}{q}[m+k]\\
&\hspace*{1cm} + q^{\binom{j-1-m}{2}} \Pqbin{k+m}{k+m-j+1,m,j-1-m}{q}([m+k]-[j-1])\Big)\psi_N^{m+k}\\
&\hspace*{1.5cm} +q^{\binom{0}{2}} \Pqbin{k+j-1}{k,j-1,0}{q}[k+j-1]\psi_N^{k+j}\bigg).
\end{align*}
By inspection, coefficient of $\psi_N^k$ is
\begin{align*}
\frac{q^{\binom{j-1}{2}}([k]-[j-1])}{[j]} \Pqbin{k}{k-j+1,0,j-1}{q}&=\frac{q^{j-1+\binom{j-1}{2}}[k-j+1]}{[j]}\Pqbin{k}{k-j+1,0,j-1}{q}\\
&=q^{\binom{j}{2}}\Pqbin{k}{k-j,0,j}{q}.
\end{align*}
 The coefficient of $\psi_N^{j+k}$ simplifies to
$$\Pqbin{k+j}{k,j,0}{q}=q^{\binom{0}{2}}\Pqbin{k+j}{k,j,0}{q}.$$
 For $1\leq m\leq j-1$, the coefficient of $\psi_N^{m+j}$ is 
\begin{align*}
\frac{q^{\binom{j-m-1}{2}}}{[j]} & \frac{[k+m]!}{[k+m-j]![m-1]![j-1-m]!}\Big(\frac{q^{j-m-1}}{[j-m]}+\frac{[m+k]-[j-1]}{[m][k+m-j+1]}\Big)\\
&\hspace*{1cm}=\frac{q^{\binom{j-m-1}{2}}}{[j]}  \frac{[k+m]!}{[k+m-j]![m-1]![j-1-m]!}\Big(\frac{q^{j-m-1}}{[j-m]}+\frac{q^{j-1}}{[m]}\Big)\\
&\hspace*{1cm}=\frac{q^{\binom{j-m}{2}}}{[j]}  \frac{[k+m]!}{[k+m-j]![m]![j-m]!}\Big([m]+q^{m}[j-m]\Big)\\
&\hspace*{1cm}=q^{\binom{j-m}{2}}\Pqbin{k+m}{k+m-j,m,j-m}{q},\\
\end{align*}
as desired.
\end{proof}
We get a family of $q$-binomial identities by evaluating the traces at elements $u_\mu\in \UT_N$.
\begin{corollary} For $0\leq j\leq k\leq n$ and $0\leq l\leq n-1$, 
$$q^{\binom{j}{2}+\binom{k}{2}}\Pqbin{n-l}{j}{q}\Pqbin{n-l}{k}{q} = \sum_{m=0}^{j} q^{\binom{j-m}{2}+\binom{k+m}{2}} \Pqbin{k+m}{k+m-j,m,j-m}{q}\Pqbin{n-l}{k+m}{q}.$$
\end{corollary}

%
%
\subsection{The connection to restriction problems} \label{SectionRainbowRestriction}

Following \cite{Mac} (up to a sign), we will let
$$\vphi_n^n(q)=\prod_{j=1}^{n} (q^j-1) = (q-1)^n[n]!.$$
For $0\leq k\leq n$, let
$$\vphi^n_k(q)=\frac{\vphi_n^n(q)}{\vphi_{n-k}^{n-k}(q)}=\prod_{j=0}^{k-1} (q^{n-j}-1)=(q-1)^{k}[k]!\Pqbin{n}{k}{q}.$$

Armed with core modules, we can now decompose the rainbow supercharacters with relative ease.   As in (\ref{RainbowNotation}), let
$$V^{n_-\underset{m}{\smile}n_+}\cong \underbrace{V^{n_-\smile n_+}\otimes \cdots \otimes V^{n_-\smile n_+}}_{\text{$m$ terms}}.$$

\begin{theorem} \label{RainbowRestriction} Let $N'=N\cup \{n_-,n_+\}$ with $\{n_-\}<N<\{n_+\}$.  Then
$$
\Res_{\UT_N}^{\UT_{N'}} (V^{n_-\underset{m}{\smile}n_+})\cong (q-1)^m\bigoplus_{k=0}^{m} \vphi_k^m(q) V_N^k.
$$
\end{theorem}
\begin{remark}
Note that $V_N^k=\{0\}$ when $k>|N|$. 
\end{remark}
\begin{proof}  We take traces for the proof, and induct on $m$.   If $m=0$, then 
$$\Res_{\UT_N}^{\UT_{N'}} (\chi^{n_-\underset{m}{\smile}n_+})=\Res_{\UT_N}^{\UT_{N'}} (\chi^\emptyset)=\chi^\emptyset=\vphi_0^0(q)\psi_N^0.$$  
For $m=1$, by Lemma \ref{SimpleRestrictionTensor} (a) and Corollary \ref{V^KRealization},
\begin{align*}
\Res_{\UT_N}^{\UT_{N'}} (\chi^{n_-\smile n_+}) &= (q-1)\Big(\chi^\emptyset+\sum_{j\in N}\Res_{\UT_N}^{\UT_{N\cup\{n_+\}}} (\chi^{j\smile n_+})\Big)\\
&= (q-1)\Big(\psi_N^0+ \vphi_1^1(q) \psi_N^1\Big).
\end{align*}

If $m>1$, then 
\begin{align*}
\Res_{\UT_N}^{\UT_{N'}} (\chi^{n_-\underset{m}{\smile}n_+})
&=\Res_{\UT_N}^{\UT_{N'}} (\chi^{n_-\smile n_+})\otimes \Res_{\UT_N}^{\UT_{N'}} (\chi^{n_-\underset{m-1}{\smile}n_+})\\
&= (q-1)(\psi_N^0+ (q-1)\psi_N^1)\otimes(q-1)^{m-1}\sum_{k=0}^{m-1} \vphi_k^{m-1}(q) \psi_N^k,
\end{align*}
by induction.
By Theorem \ref{V^kTensor},
\begin{align*}
\Res_{\UT_N}^{\UT_{N'}} (\chi^{n_-\underset{m}{\smile}n_+})
&=(q-1)^{m}\sum_{k=0}^{m-1} \vphi_k^{m-1}(q) (\psi_N^0\otimes \psi_N^k + (q-1)\psi_N^1\otimes \psi_N^k)\\
&=(q-1)^{m}\sum_{k=0}^{m-1} \vphi_k^{m-1}(q)(\psi_N^k +(q-1) [k]\psi_N^k+(q-1)[k+1]\psi_N^{k+1})\\
&=(q-1)^{m}\sum_{k=0}^{m-1} \vphi_k^{m-1}(q)(q^k\psi_N^k+(q-1)[k+1]\psi_N^{k+1}).
\end{align*}
Collect coefficients to get
\begin{align*}
\Res_{\UT_N}^{\UT_{N'}} (\chi^{n_-\underset{m}{\smile}n_+})
&=(q-1)^m\sum_{k=0}^{m} \Big( (q-1)\vphi_{k-1}^{m-1}(q)[k] +\vphi_k^{m-1}(q)q^k\Big)  \psi_N^k\\
&=(q-1)^{m}\sum_{k=0}^{m}\vphi_{k-1}^{m-1}(q) \Big( (q^k-1+q^k(q^{m-k}-1) \Big)  \psi_N^k\\
&=(q-1)^{m}\sum_{k=0}^{m} \vphi_k^m(q)  \psi_N^k,
\end{align*}
as desired.
\end{proof}

Combine Theorem \ref{RainbowRestriction} with Proposition \ref{V^kDecomposition} to get the following corollary.

\begin{corollary} \label{RainbowToSupercharacters}
Let $N'=N\cup \{n_-,n_+\}$ with $\{n_-\}<N<\{n_+\}$.  Then
$$
\Res_{\UT_N}^{\UT_{N'}} (V^{n_-\underset{m}{\smile}n_+})\cong (q-1)^{m} \bigoplus_{\lambda\in \scS_N\atop |\lambda|\leq m} \Big(q^{\nst_\lambda^\lambda} \sum_{k=|\lambda|}^{m} \vphi_k^m(q) \Pqbin{\cP(\uncr{\lambda})}{k-|\lambda|}{q}\Big) V^\lambda.
$$
\end{corollary}

\section{An onion peel approach to restriction}\label{SectionOnionPeel}

This section studies the next generalization where $N'=\{n_{--},n_-,n_+,n_{++}\}\sqcup N$ satisfies
\begin{itemize}
\item $N\neq \emptyset$.
\item $n_{--}\leq n_{-}\leq n_{+}\leq n_{++}$,
\item $n_{--}<N<n_{++}$,
\item $n_-<N$ implies $n_-=n_{--}$ and $N<n_+$ implies $n_+=n_{++}$.
\end{itemize}
We let $N=N_<\cup N_=\cup N_>$ where $\{n_{--}\}\leq N_<\leq \{n_-\}\leq N_=\leq \{n_+\}\leq N_>\leq \{n_{++}\}$.  This allows us to also use the conventions $N_\leq =N_<\cup N_=$ and $N_\geq =N_=\cup N_>$.  
We are interested in the multisets of the form
\begin{equation}\label{DoubleRainbowSetup}
\begin{tikzpicture}[scale=.5,baseline=0]
	\foreach \x in {0,1,2,4,5,6,7,9,10,11,12,14,15}
		\node (\x) at (\x,0) [inner sep =-1pt] {$\bullet$};
	\foreach \x in {3,8,13}
		\node at (\x,0) {$\cdots$};
	\foreach \d in {0,5,20}
		\draw (0) to [out=-45-\d,in=-135+\d] (15);	
	\foreach \i in {0,.1,.2}
		\node at (7.5,-1.5+\i) {$\cdot$};
	\foreach \i in {0,.2,.4}
		\node at (7.5,-4+\i) {$\cdot$};	
	\foreach \d in {0,5,30}
		\draw (5) to [out=-45-\d,in=-135+\d] (10);
	\node at (0,.5) {$\scs n_{--}$};	
	\node at (5,.5) {$\scs n_{-}$};	
	\node at (10,.5) {$\scs n_{+}$};	
	\node at (15,.5) {$\scs n_{++}$};	
	\draw[|-|] (1,.5) -- node [above] {$N_<$} (4,.5);	
	\draw[|-|] (1,1.5) -- node [above, pos=.35] {$N_\leq\cup \{n_-\}$} (9,1.5);
	\draw[|-|] (6,.5) -- node [above] {$N_=$} (9,.5);
	\draw[|-|] (6,2) -- node [above] {$N_\geq\cup \{n_+\}$} (14,2);
	\draw[|-|] (11,.5) -- node [above] {$N_>$} (14,.5);
\end{tikzpicture}\ .
\end{equation}
As with the rainbow case, this ``double" rainbow comes from a supercharacter restriction problem.  That is, we note that if $N''=N_{--}\cup N_-\cup N_+\cup N_{++}\cup N$ with $|N_{--}|=|N_{++}|$, $|N_-|=|N_+|$ and 
$$N_{--}<N_{<}<N_-<N_=<N_+<N_><N_{++},$$
then if $w:N_{--}\cup N_- \rightarrow N_+\cup N_{++}$ is the unique bijection such that $i<j$ if and only if $w(j)<w(i)$, then
\begin{align*}
\Res_{\UT_N}^{\UT_{N''}}&(V^{\{i\smile w(i)\mid i\in N_{--}\cup N_-\}})\\
&\cong q^{2\binom{|N_{--}|}{2}+2\binom{|N_-|}{2}+2|N_{--}|(|N_-|-1)}\Res_{\UT_{N}}^{\UT_{N'}}\Big(V^{n_{--}\underset{|N_{--}|}{\smile}n_{++}}\otimes V^{n_{-}\underset{|N_{-}|}{\smile}n_{+}}\Big)
\end{align*}
is easy to establish by taking traces and comparing the character values on $\UT_N$.

\begin{remark}
The generic case is where $N_<$, $N_=$ and $N_>$ are all nonempty; however, the notation allows for several degenerate cases.  In particular, if $N_<\cup N_>=\emptyset$ or $N_==\emptyset$, then we reduce to the rainbow case of Section \ref{SectionRainbow}.
\end{remark}

\subsection{Right-endpoint submodules}

Instead of clumping together the $V_N^K$ as in Section \ref{SectionColumnCardinality}, we can instead decompose them further, so
\begin{equation}\label{V^KToV^hook}
V_N^K=\bigoplus_{J\subseteq N}V_N^{K\hookleftarrow J},
\end{equation}
where
\begin{align}
V_N^{K\hookleftarrow J} 
&= \CC\spanning\{\overline{av_\lambda b}\mid a\in B_N, b\in B_K, \epl{\lambda}\subseteq K, \epr{\lambda}=J\}\notag\\
&\cong \bigoplus_{\lambda\in \scS_{N}\atop \epl{\lambda}\subseteq K, \epr{\lambda}=J}   q^{\nst_\lambda^\lambda+\nst_{K-\epl{\lambda}}^\lambda}V^\lambda. \label{V^hookDecomposition}
\end{align}
Note that $V^{K\hookleftarrow J}$ directly specifies what the right endpoints will be in the set partitions appearing in the decomposition.  We can additionally specify the left endpoints by considering subsets of the appropriate size of $K$.

\begin{proposition} \label{LeftRightEndpointModules}
For $J,K\subseteq N$,
$$V_N^{K\hookleftarrow J}=\bigoplus_{I\subseteq K\atop |I|=|J|} q^{\wt_J^\uparrow(K-I)-\wt_I^\uparrow(K-I)}V_N^{I\hookleftarrow J}.$$
\end{proposition}
\begin{proof}
Note that if $\lambda\in \scS_N$, then
$$\nst_{K-I}^\lambda=\wt_{\epr{\lambda}}^\uparrow(K-I)-\wt_{\epl{\lambda}}^\uparrow(K-I)=\wt_{J}^\uparrow(K-I)-\wt_{I}^\uparrow(K-I).$$
Now use (\ref{V^hookDecomposition}).
\end{proof}

In the following sections, we will make use of the ``flipped" modules (as in (\ref{ModuleInvolution})) coming from the action (\ref{RightAction})
$$\overrightarrow{V}_N^K\cong \dagger\Big(\overleftarrow{V}_N^{w_0(K)}\Big)\qquad \text{and}\qquad  \overrightarrow{V}_N^{J\hookrightarrow K}\cong \dagger\Big(\overleftarrow{V}_N^{w_0(K)\hookleftarrow w_0(J)}\Big),$$
where $w_0:N\rightarrow N$ is the order flipping involution as in (\ref{SetPartitionInvolution}).  Note that we can use Proposition \ref{LeftRightModules} to decompose these alternate modules using the $\dagger$-versions of results already proved.

\subsection{Peel modules}

Let $N'=\{n_{--},n_{-},n_+,n_{++}\}\cup N$ be as in (\ref{DoubleRainbowSetup}).  
For $0\leq b\leq \min\{|N_<|,|N_>|\}$ and $b\leq f\leq |N_<\cup N_>|$, let
$$V_{N_=\subseteq N}^{(b;f)}\cong\bigoplus_{{L\subseteq N_\geq,R\subseteq N\atop F\subseteq N_<\cup(N_>-R)}\atop |F|=f, |R\cap N_>|=b}q^{\wt_{R}^\uparrow(F\cap N_>)} \overleftarrow{V}_N^{(F\cap N_<)\hookleftarrow R}\otimes  \overrightarrow{V}_{N}^{L\hookrightarrow (F\cap N_>)}.$$

To understand what is going on with these modules, it is perhaps easiest to consider an example.  Let $N=\{1,2,\ldots, 12\}$ and $N_==\{4,\ldots, 8\}$.  Then we are considering a subset of the matrices of the form
$$\left[\begin{array}{ccc|ccccc|cccc} 
0 & & & & & & & & & &  \\
\ast & 0 & & & & & & & & &\\
\ast & \ast & 0 & & & & & & & &\\ \hline
\ast & \ast & \ast & 0  & & & & & & &\\
\ast & \ast & \ast & 0 & 0 & & & & & &\\
\ast & \ast & \ast & 0 & 0 & 0 & & & & &\\
\ast & \ast & \ast & 0 & 0 & 0 & 0 & & & &\\
\ast & \ast & \ast & 0 & 0 & 0 & 0 & 0 & & &\\ \hline
\ast & \ast & \ast & \ast & \ast & \ast & \ast & \ast & 0 & &\\
\ast & \ast & \ast & \ast & \ast & \ast & \ast & \ast & \ast & 0 &\\
\ast & \ast & \ast & \ast & \ast & \ast & \ast & \ast & \ast & \ast & 0\\
\ast & \ast & \ast & \ast & \ast & \ast & \ast & \ast & \ast & \ast & \ast & 0
\end{array}\right]$$
These matrices are not closed under left multiplication by $\UT_N$, so we instead let $u\in \UT_N$ act by
$$
u\lact\left[\begin{array}{ccc} 
0 & &   \\
\ast & 0 & \\
\ast & \ast & 0 \\ \hline
\ast & \ast & \ast\\
\ast & \ast & \ast\\
\ast & \ast & \ast \\
\ast & \ast & \ast\\
\ast & \ast & \ast\\ \hline
\ast & \ast & \ast \\
\ast & \ast & \ast\\
\ast & \ast & \ast\\
\ast & \ast & \ast
\end{array}\right] \otimes 
u\ract\left[\begin{array}{ccc|ccccc|cccc} 
0 & 0 & 0 & \ast & \ast & \ast & \ast & \ast & 0 & &\\
0 & 0 & 0 & \ast & \ast & \ast & \ast & \ast & \ast & 0 &\\
0 & 0 & 0 & \ast & \ast & \ast & \ast & \ast & \ast & \ast & 0\\
0 & 0 & 0 & \ast & \ast & \ast & \ast & \ast & \ast & \ast & \ast & 0
\end{array}\right] $$
In $V_{N_=\subseteq N}^{(b;f)}$ we further specify 
$$b=\mathrm{rank} \left[\begin{array}{ccc} 
\ast & \ast & \ast \\
\ast & \ast & \ast\\
\ast & \ast & \ast\\
\ast & \ast & \ast
\end{array}\right],$$
and bound the rank of the whole matrix by $f$.

\begin{example}
In the degenerate cases,   
\begin{itemize}
\item If $N_==\emptyset$, then we do not see $n_-=n_+$, but they identify a distinguished spot (which becomes important when decomposing into supercharacters).  
\item If $N_<\cup N_>=\emptyset$, then $b=f=0$ and these modules are uninteresting (aka trivial).
\item If $N_<=\emptyset$ or $N_>=\emptyset$ (but not both), then $b=0$. If, for example, $N_<\neq \emptyset$, the module is isomorphic to
\begin{equation}\label{DegenerateCase3}
\bigoplus_{F\subseteq N_<\atop |F|=f} V_N^F.
\end{equation}
\item If $N_\leq=\emptyset$ or $N_\geq =\emptyset$, then $b=0$ and we get a core module
$$V_{\emptyset\subseteq N}^{(0;f)}\cong V_N^f.$$
\end{itemize}
\end{example}  

By construction, it is easy to decompose the peel modules into supercharacters.  For $\nu\in\scS_N$, let
\begin{equation}\label{ArcBlocks}
{}_\alpha\nu_\beta=\{i\smile j\in \nu\mid i\in N_\alpha, j\in N_\beta\}.
\end{equation}
Recall that $\cP({\uncr{\nu}})$ is a poset on the blocks $\bl(\uncr{\nu})$ of $\uncr{\nu}$.  For $K\subseteq N$ let 
\begin{align*}
\bl_{\epr{K}}(\nu) &= \{\{i_1\smile\cdots\smile i_\ell\in \bl(\nu)\mid i_\ell\in K\}\\
\bl_{\epl{K}}(\nu) &= \{\{i_1\smile\cdots\smile i_\ell\in \bl(\nu)\mid i_1\in K\}
\end{align*} 
or the set of blocks with right-most (resp. left-most) endpoints in $K$.

Let 
$$\psi_{N_=\subseteq N}^{(b;f)}=\tr(\cdot, V_{N_=\subseteq N}^{(b;f)}).$$

\begin{proposition}\label{PeelDecomposition} For $0\leq b\leq \min\{|N_<|,|N_>|\}$ and $b\leq f\leq |N_<\cup N_>|$, 
$$V_{N_=\subseteq N}^{(b;f)}\cong \bigoplus_{\nu\in \scS_N,|\nu|\leq f\atop |{}_<\nu_{>}|=b,|{}_=\nu_=|=\emptyset} q^{\nst_\nu^\nu} \Pqbin{\cP(\uncr{\nu})}{f-|\nu|\subseteq \bl_{\epr{N_<}}(\uncr{\nu})\cup \bl_{\epl{N_>}}(\uncr{\nu}),0}{q} V^\nu.$$
\end{proposition}

\begin{proof}  Take traces and use (\ref{V^hookDecomposition}) to get
\begin{align*}
\psi_{N_=\subseteq N}^{(b;f)} &=\hspace{-.5cm}
\sum_{{L\subseteq N_\geq,R\subseteq N\atop F\subseteq N_<\cup(N_>-R)}\atop |F|=f,|R\cap N_>|=b}\hspace{-.5cm} q^{\wt_{R}^\uparrow(F\cap N_>)}\hspace{-.5cm}\sum_{\lambda\in \scS_N\atop \epl{\lambda}\subseteq F\cap N_<,\epr{\lambda}=R} \hspace{-.5cm}q^{\nst_\lambda^\lambda+\nst_{F\cap N_<-\epl{\lambda}}^\lambda} \chi^{\lambda}\odot \hspace{-.5cm}\sum_{\mu\in \scS_{N}\atop \epl{\mu}=L,\epr{\mu}\subseteq F\cap N_>}\hspace{-.5cm}q^{\nst_\mu^\mu+\nst_{F\cap N_>-\epr{\mu}}^\mu} \chi^\mu\\
&=\sum_{\nu\in \scS_N,|\nu|\leq f\atop |{}_<\nu_{>}|=b,|{}_=\nu_=|=\emptyset}q^{\nst_\nu^\nu}\Big(\sum_{F\subseteq (N_<-\epl{\nu})\times (N_>-\epr{\nu})\atop |F|+|\nu|=f} q^{\nst_F^\nu}\Big)\chi^\nu\\
&=\sum_{\nu\in \scS_N,|\nu|\leq f\atop |{}_<\nu_{>}|=b,|{}_=\nu_=|=\emptyset} q^{\nst_\nu^\nu}\Pqbin{\cP(\uncr{\nu})}{f-|\nu|\subseteq \bl_{\epr{N_<}}(\uncr{\nu})\cup \bl_{\epl{N_>}}(\uncr{\nu}),0}{q} \chi^\nu,
\end{align*}
as desired.
\end{proof}

The main theorem of this section decomposes the double rainbow module in terms of peel modules.  

\begin{theorem}\label{TensorToPeel}
Let $N'=\{n_{--},n_{-},n_+,n_{++}\}\cup N$ be as in (\ref{DoubleRainbowSetup}), and fix $m,\ell\in \ZZ_{\geq 0}$.  If $\mu=\{n_{--}\underset{m}{\smile}n_{++},n_-\underset{\ell}{\smile}n_+\}\in \scM_{N'}$, then as a $\UT_N$-module
\begin{align*}
V^\mu&
\cong  q^{\nst_{\epl{\mu}\cup\epr{\mu}}^\mu}\hspace{-1cm} \bigoplus_{0\leq b\leq \min\{|N_<|,|N_>|\}\atop b\leq f\leq \min\{m,|N_<\cup N_>|\}}\hspace{-1cm} (q-1)^{f}q^{(m-f)b} \vphi_f^m(q)V^{(b;f)}_{N_=\subseteq N}\otimes V^{n_-\underset{m-f+\ell}{\smile}n_+}.
\end{align*}
\end{theorem}

Unfortunately, the proof of Theorem \ref{TensorToPeel} lacks a certain elegance; rather, it employs the following slight generalization of Theorem \ref{RainbowRestriction} repeatedly.  Essentially, we want to know how the relevant modules decompose when there are disallowed positions.


\begin{lemma} \label{LeftInterferenceRestriction} Let $k_-,k_+\in N$ with $k_-<k_+$ and let $\bar{K}$ be the interval in $N$ between $k_-$ and $k_+$.  Let $K\subseteq \bar{K}$ and $\lambda\in\scM_N$ with $\lambda_{\bar{K}}=\emptyset$ and  $\epl{X}=(\epl{\lambda}\cap K)$. Then as a $\UT_K$-module
$$V^\lambda\otimes V^{k_{-}\underset{\ell}{\smile} k_{+}}\cong  V^\lambda\otimes q^{\ell|\bar{K}-K|} (q-1)^{\ell}\hspace{-.2cm} \bigoplus_{J\subseteq K-\epl{X}\atop |J|\leq \ell} \hspace{-.2cm}  q^{\wt_{\epl{X}}^\downarrow(J)+(\ell-|J|)|\epl{X}|} \vphi_{|J|}^\ell(q)   V_K^J.
$$
\end{lemma}

\begin{proof}  Note that if $j\smile k$ is an arc with $j,k\in K$, then
$$\Res_{\UT_K}^{\UT_{\bar{K}}}(\chi^{j\smile k})=q^{\nst_{\bar{K}-K}^{j\smile k}}\chi^{j\smile k}$$
follows from Proposition \ref{SupercharacterFormula}.
 Since restriction is transitive, we see that 
$$\Res^{\UT_{\{k_-,k_+\}\cup \bar{K}}}_{\UT_{K}}(\chi^{k_{-}\underset{\ell}{\smile} k_{+}})=q^{\ell|\bar{K}-K|}\Res^{\UT_{\{k_-,k_+\}\cup K}}_{\UT_{K}}(\chi^{k_{-}\underset{\ell}{\smile} k_{+}}).$$
So WLOG, we will assume that $\bar{K}=K$.

Induct on $\ell$, and let $L=\epl{X}$.  Let $\ell=1$.  If $L\cap \epl{\mu}\neq \emptyset$, then $\chi^\lambda(u_\mu)=0$ implies the result holds in this case. WLOG, assume $L\cap \epl{\mu}= \emptyset$.  Then the LHS gives
\begin{equation*}
\chi^\lambda(u_\mu)\chi^{k_-\smile k_+}(u_\mu)^1=
(q-1)^{|\lambda|+1}q^{\dim(\lambda)+(|K|-|\mu|)-\nst_{\mu}^\lambda}.
\end{equation*}
On the other hand, by Proposition \ref{V^KValue}, the RHS gives
\begin{align*}
(q-1) \chi^\lambda(u_\mu)&\Big(q^{|L|}\vphi_0^1(q)\psi_K^\emptyset(u_\mu)+\sum_{j\in K-L} q^{\wt_{L}^\downarrow(j)}\vphi_1^1(q) \psi_K^{\{j\}}(u_\mu)\Big)\\ 
&= (q-1)^{1+|\lambda|} q^{\dim(\lambda)-\nst_\mu^\lambda} \Big(q^{|L|}+\sum_{j\in K-L-\epl{\mu}} (q-1) q^{\wt_L^\downarrow(j)} q^{\wt_K^\uparrow(j)-\wt_{\epl{\mu}}^\uparrow(j)}\Big)\\
&= (q-1)^{1+|\lambda|} q^{\dim(\lambda)-\nst_\mu^\lambda} \Big(q^{|L|}+(q-1)\sum_{j\in K-L-\epl{\mu}}  q^{\wt_L^\downarrow(j)-\wt_{\epl{\mu}}^\uparrow(j)} q^{\wt_K^\uparrow(j)}\Big)\\
&=(q-1)^{1+|\lambda|} q^{\dim(\lambda)-\nst_\mu^\lambda} \Big(q^{|L|}+(q-1)[|K-L-\epl{\mu}|]q^{(|K|-|\epl{\mu}|-1)-(|K|-|L|-|\epl{\mu}|)+1}\Big)\\
&=(q-1)^{1+|\lambda|} q^{\dim(\lambda)-\nst_\mu^\lambda} \Big(q^{|L|}+(q^{|K-L-\epl{\mu}|}-1)q^{|L|}\Big)\\
&=(q-1)^{1+|\lambda|} q^{\dim(\lambda)-\nst_\mu^\lambda} q^{|K-\epl{\mu}|}.
\end{align*}
 Suppose $\ell>1$.  Then by induction and Corollary \ref{V^KRealization}, 
\begin{align*}
\chi^\lambda\odot \chi^{k_{-}\underset{\ell}{\smile} k_{+}}&=
\chi^\lambda\odot \chi^{k_{-}\underset{\ell-1}{\smile} k_{+}}\odot \chi^{k_{-}\smile k_{+}}\\
&=(q-1)^{\ell-1}\sum_{J\subseteq K-L\atop |J|\leq \ell-1} q^{\wt_{L}^\downarrow(J)+(\ell-1-|J|)|L|} \vphi_{|J|}^{\ell-1}(q)   \chi^\lambda\odot \psi_K^J \odot \chi^{k_{-}\smile k_{+}}\\
&=(q-1)^{\ell-1} \sum_{J\subseteq K-L\atop |J|\leq \ell-1}(q-1)^{-|J|} q^{\wt_{L}^\downarrow(J)+(\ell-1-|J|)|L|}\vphi_{|J|}^{\ell-1}(q)   \chi^\lambda\odot  \bigodot_{j\in J} \chi^{j\smile k_+}\odot\chi^{k_{-}\smile k_{+}}.
\end{align*}
By adding $\{j\smile k_+\mid j\in J\}$ as a multiset to $\lambda$ and then applying induction, 
\begin{align*}
\Big( \chi^\lambda\odot &\bigodot_{j\in J} \chi^{j\smile k_+}\Big)\odot \chi^{k_{-}\smile k_{+}}\\
 &=\Big(  \chi^\lambda\odot\bigodot_{j\in  J} \chi^{j\smile k_+}\Big) \odot  (q-1)\Big(q^{|J|+|L|}\vphi_0^1(q)\psi_K^\emptyset + \sum_{i\in K-L-J} q^{\wt_J^\downarrow(i)+\wt_L^\downarrow(i)}\vphi_1^1(q)\psi_{K}^{\{i\}}\Big)\\
 &= \chi^\lambda\odot (q-1)\Big( (q-1)^{|J|}q^{|J|+|L|}\psi_K^J+\sum_{i\in K-L-J} (q-1)^{|J|+1}q^{\wt_J^\downarrow(i)+\wt_L^\downarrow(i)}\psi_K^{J\cup \{i\}}\Big).
\end{align*}
Thus, the coefficient of $\psi_K^\emptyset$ is 
$$(q-1)^{\ell}q^{(\ell-1)|L|+|L|}\vphi_0^{\ell-1}(q)=(q-1)^{\ell}q^{(\ell)|L|}\vphi_0^\ell(q).$$
  On the opposite extreme, the coefficient of $\psi_K^J$ for some $\ell$-subset $J\subseteq K-L$ is 
\begin{align*}
(q-1)^{\ell+1} \sum_{i\in J}q^{\wt_L^\downarrow(J-\{i\})+\wt_{J-\{i\}}^\downarrow(i)+\wt_{L}^\downarrow(i)}\vphi_{\ell-1}^{\ell-1}(q)
&=(q-1)^{\ell+1}q^{\wt_L^\downarrow(J)}\vphi_{\ell-1}^{\ell-1}(q)
\sum_{i\in J} q^{\wt_{J-\{i\}}^\downarrow(i)}\\
&=(q-1)^{\ell} q^{\wt_L(J)}\vphi_\ell^\ell(q).
\end{align*}
For the intermediate cases, fix $J\subseteq K-L$ with $1\leq |J|\leq \ell-1$.  Then the coefficient of $\psi_K^J$ is 
\begin{align*}
(q-1)^{\ell}  &\Big(q^{\wt_{L}^\downarrow(J)+(\ell-1-|J|)|L|+|L|+|J|} \vphi_{|J|}^{\ell-1}(q)+\sum_{i\in J}(q-1) q^{\wt_{L}^\downarrow(J)+\wt_{J-\{i\}}^\downarrow(i)+(\ell-|J|)|L|} \vphi_{|J|-1}^{\ell-1}(q)\Big)\\
&=(q-1)^{\ell} q^{\wt_L^\downarrow(J)+(\ell-|J|)|L|}\vphi_{|J|-1}^{\ell-1}(q) \Big(q^{|J|}(q^{\ell-|J|}-1)+(q^{|J|}-1)\Big)\\
&=(q-1)^{\ell} q^{\wt_L^\downarrow(J)+(\ell-|J|)|L|}\vphi_{|J|}^{\ell}(q),
\end{align*}
as desired.
\end{proof}

As mentioned above, the proof of Theorem \ref{TensorToPeel} uses Lemma \ref{LeftInterferenceRestriction} repeatedly.  The result is somewhat technical, so we give a brief outline of the proof before we begin.  The first observation is that it suffices to consider the case when $\ell=0$. We can visualize the steps as follows.  We first show that 
\begin{equation}\label{MainProofStep1}
\begin{tikzpicture}[scale=.25,baseline=0]
	\foreach \x in {0,1,2,4,5,6,7,9,10,11,12,14,15}
		\node (\x) at (\x,0) [inner sep =-1pt] {$\bullet$};
	\foreach \x in {3,8,13}
		\node at (\x,0) {$\scs\cdots$};
	\draw (0) to [out=-45,in=-135] (15);	
	\node at (7.5,-4) {$\scs m$};	
	\node at (0,.75) {$\scscs n_{--}$};	
	\node at (5,.75) {$\scscs n_{-}$};	
	\node at (10,.75) {$\scscs n_{+}$};	
	\node at (15,.75) {$\scscs n_{++}$};
\end{tikzpicture}\approx
\begin{tikzpicture}[scale=.25,baseline=0]
	\foreach \x in {0,1,2,4,5,6,7,9,10,11,12,14,15}
		\node (\x) at (\x,0) [inner sep =-1pt] {$\bullet$};
	\foreach \x in {3,8,13}
		\node at (\x,0) {$\scs\cdots$};
	\draw (1,-.5) to [out=-45,in=-135] (15);
	\draw (4,-.5) to [out=-45,in=-135] (15);
	\draw (1,-.5) -- (4,-.5);	
	\node at (5,-2) {$\scs l$};	
	\node at (0,.75) {$\scscs n_{--}$};	
	\node at (5,.75) {$\scscs n_{-}$};	
	\node at (10,.75) {$\scscs n_{+}$};	
	\node at (15,.75) {$\scscs n_{++}$};
\end{tikzpicture}\odot
\begin{tikzpicture}[scale=.25,baseline=0]
	\foreach \x in {0,1,2,4,5,6,7,9,10,11,12,14,15}
		\node (\x) at (\x,0) [inner sep =-1pt] {$\bullet$};
	\foreach \x in {3,8,13}
		\node at (\x,0) {$\scs\cdots$};
	\draw (5) to [out=-45,in=-135] (15);	
	\node at (10.5,-3) {$\scs m-l$};	
	\node at (0,.75) {$\scscs n_{--}$};	
	\node at (5,.75) {$\scscs n_{-}$};	
	\node at (10,.75) {$\scscs n_{+}$};	
	\node at (15,.75) {$\scscs n_{++}$};
\end{tikzpicture},
\end{equation}
where term labeled with $l$ is the sum over all $\overleftarrow{V}_N^K$ where $K\subseteq N_<$ is an $l$-subset.  Then we note that
\begin{equation}\label{MainProofStep2}
\begin{tikzpicture}[scale=.25,baseline=0]
	\foreach \x in {0,1,2,4,5,6,7,9,10,11,12,14,15}
		\node (\x) at (\x,0) [inner sep =-1pt] {$\bullet$};
	\foreach \x in {3,8,13}
		\node at (\x,0) {$\scs\cdots$};
	\draw (5) to [out=-45,in=-135] (15);	
	\node at (10.5,-3) {$\scs m-l$};	
	\node at (0,.75) {$\scscs n_{--}$};	
	\node at (5,.75) {$\scscs n_{-}$};	
	\node at (10,.75) {$\scscs n_{+}$};	
	\node at (15,.75) {$\scscs n_{++}$};
\end{tikzpicture}\approx \begin{tikzpicture}[scale=.25,baseline=0]
	\foreach \x in {0,1,2,4,5,6,7,9,10,11,12,14,15}
		\node (\x) at (\x,0) [inner sep =-1pt] {$\bullet$};
	\foreach \x in {3,8,13}
		\node at (\x,0) {$\scs\cdots$};
	\draw (5) to [out=-45,in=-135] (10);	
	\node at (7.5,-2) {$\scs m-l-r$};	
	\node at (0,.75) {$\scscs n_{--}$};	
	\node at (5,.75) {$\scscs n_{-}$};	
	\node at (10,.75) {$\scscs n_{+}$};	
	\node at (15,.75) {$\scscs n_{++}$};
\end{tikzpicture}\odot 
\begin{tikzpicture}[scale=.25,baseline=0]
	\foreach \x in {0,1,2,4,5,6,7,9,10,11,12,14,15}
		\node (\x) at (\x,0) [inner sep =-1pt] {$\bullet$};
	\foreach \x in {3,8,13}
		\node at (\x,0) {$\scs\cdots$};
	\draw (11,-.5) -- (14,-.5);
	\draw (5) to [out=-45,in=-135] (11,-.5);
	\draw (5) to [out=-45,in=-135] (14,-.5);	
	\node at (12,-1) {$\scs r$};	
	\node at (0,.75) {$\scscs n_{--}$};	
	\node at (5,.75) {$\scscs n_{-}$};	
	\node at (10,.75) {$\scscs n_{+}$};	
	\node at (15,.75) {$\scscs n_{++}$};
\end{tikzpicture},
\end{equation}
where the term labeled with $r$ is the sum over all $\overrightarrow{V}_{N_\geq}^K$ where $K\subseteq N_>$ is an $r$-subset.  Now the terms labeled with $l$ and the $r$ form the peel modules and the left-overs go to rainbow modules.  

\begin{proof}[Proof of Theorem \ref{TensorToPeel}]  We prove the generic case where $N_<$, $N_=$ and $N_>$ are all nonempty.  The general case follows by omitting the steps (\ref{MainProofStep1}) and/or (\ref{MainProofStep2}) as necessary.

As usual, we take traces.  Note that since tensoring by $\chi^{n_-\underset{\ell}{\smile} n_+}$ is straightforward, it suffices to show that as elements of $\scf(\UT_N)$,
\begin{equation*}
\chi^{n_{--}\underset{m}{\smile} n_{++}}= q^{m\#\{n_-,n_+\}} \sum_{0\leq b\leq \min\{|N_<|,|N_>|\}\atop b\leq f\leq \min\{m,|N_<\cup N_>|\}}(q-1)^{f}q^{(m-f)b} \vphi_f^m(q)\psi^{(b;f)}_{N_=\subseteq N}\odot \chi^{n_-\underset{m-f}{\smile}n_+},
\end{equation*}
where we also note that $\nst_{\epl{\mu}\cup\epr{\mu}}^\mu=m\#\{n_-,n_+\}$ in the generic case.

By Lemma \ref{LeftInterferenceRestriction}  with $\bar{K}=N\cup\{n_-,n_+\}$ and $\lambda=\emptyset$,
$$\chi^{n_{--}\underset{m}{\smile} n_{++}}=q^{m\#\{n_-,n_+\}} (q-1)^{m}\sum_{L_1\subseteq N\atop |L_1|\leq m}\vphi_{|L_1|}^m(q)\psi^{L_1}_{N}.
$$
Each subset $L_1\subseteq N$, breaks up into $L_1^<=L_1\cap N_<\subseteq N_<$ and $L_1^\geq=L_1\cap N_\geq\subseteq N_\geq$. Use this observation and then Lemma \ref{LeftInterferenceRestriction}  ``in reverse" on $N_\geq$ with $\bar{K}=\{n_+\}\cup N_\geq$ and $\lambda=\emptyset$ to get
\begin{align*}
\chi^{n_{--}\underset{m}{\smile} n_{++}}&=q^{m\#\{n_-\}}\hspace{-.3cm}\sum_{L_1^<\subseteq N_<\atop |L_1^<|\leq m} (q-1)^{|L_1^<|}q^{|L_1^<|\#\{n_+\}}\vphi_{|L_1^<|}^{m}(q)\psi_N^{L_1^<}\\
&\hspace*{1cm}\odot q^{(m-|L_1^<|)\#\{n_+\}} \hspace{-.3cm} \sum_{L_1^\geq\subseteq N_\geq\atop |L_1^\geq|\leq m-|L_1^<|} \hspace{-.3cm}  (q-1)^{m-|L_1^<|} \vphi_{|L_1^\geq|}^{m-|L_1^<|}(q)\psi_N^{L_1^\geq}\\
&=q^{m\#\{n_-\}}\hspace{-.3cm}\sum_{L_1^<\subseteq N_<\atop |L_1^<|\leq m} (q-1)^{|L_1^<|} q^{|L_1^<|\#\{n_+\}}\vphi_{|L_1^<|}^{m}(q)\psi_N^{L_1^<}\odot \chi^{n_-\underset{m-|L_1^<|}{\smile}n_{++}}\\
&=q^{m\#\{n_-\}}\hspace{-.3cm}\sum_{L_1^<\subseteq N_<,R_1\subseteq N\atop |L_1^<|\leq m} (q-1)^{|L_1^<|} q^{|L_1^<|\#\{n_+\}}\vphi_{|L_1^<|}^{m}(q) \psi_N^{L_1^<\hookleftarrow R_1}\odot \chi^{n_-\underset{m-|L_1^<|}{\smile}n_{++}},
\end{align*}
where the last equality is by (\ref{V^KToV^hook}).

We will now begin to use the action $\lact$ in addition to $\ract$.  We will therefore use the arrows $\overleftarrow{\psi}$ and $\overrightarrow{\psi}$ to indicate which kind of the module the trace is coming from.

Let $R_1^\geq=R_1\cap N_\geq$.  By the $\ract$-version of Lemma \ref{LeftInterferenceRestriction}  with $K=N_\geq$ and $\epr{X}=R_1^\geq$,
\begin{align*}
\overleftarrow{\psi}_N^{L_1^<\hookleftarrow R_1}
\odot \chi^{n_-\underset{m-|L_1^<|}{\smile}n_{++}}&
=\overleftarrow{\psi}_N^{L_1^<\hookleftarrow R_1} \odot q^{(m-|L_1^<|)\#\{n_+\}}\\
& \cdot\sum_{R_2\subseteq N_\geq -R_1^\geq\atop |R_2|\leq m-|L_1^<|}  \hspace{-.5cm}(q-1)^{m-|L_1^<|} q^{\wt_{R_1^\geq}^\uparrow(R_2)+(m-|L_1^<|-|R_2|)|R_1^\geq|} \vphi_{|R_2|}^{m-|L_1^<|}(q) \overrightarrow{\psi}_{N}^{R_2}.
\end{align*}
Let $R_1^>=R_1^\geq \cap N_>$ and $R_1^==R_1^\geq \cap N_=$.   Break up each subset $R_2$ into $R_2^==R_2\cap N_=$ and $R_2^>=R_2\cap N_>$.  We again employ Lemma \ref{LeftInterferenceRestriction}  ``in reverse" with $K=N_=$ and $\epr{X}=R_1^=$ to get
\begin{align*}
&\overleftarrow{\psi}_N^{L_1^<\hookleftarrow R_1}
\odot \chi^{n_-\underset{m-|L_1^<|}{\smile}n_{++}}\\
&=\overleftarrow{\psi}_N^{L_1^<\hookleftarrow R_1}
\odot q^{(m-|L_1^<|)\#\{n_+\}}\hspace{-.5cm} \sum_{R_2^>\subseteq N_> -R_1^>\atop |R_2^>|\leq m-|L_1^<|}  \hspace{-.5cm} 
(q-1)^{|R_2^>|} q^{\wt_{R_1^>}^\uparrow(R_2^>)+(m-|L_1^<|-|R_2^>|)|R_1^>|}\vphi_{|R_2^>|}^{m-|L_1^<|}(q) \overrightarrow{\psi}_N^{R_2^>}
\\
&\hspace*{1cm}\odot \hspace{-.5cm}\sum_{R_2^=\subseteq N_= -R_1^=\atop |R_2^=|\leq m-|L_1^<|-|R_2^>|}\hspace{-.9cm}
(q-1)^{m-|L_1^<|-|R_2^>|}q^{\wt_{R_1^=}^\uparrow(R_2^=)+(m-|L_1^<|-|R_2^>|-|R_2^=|)|R_1^=|}\vphi_{|R_2^=|}^{m-|L_1^<|-|R_2^>|}(q) \overrightarrow{\psi}_{N}^{R_2^=}\\
&=\overleftarrow{\psi}_N^{L_1^<\hookleftarrow R_1}
\odot q^{(m-|L_1^<|)\#\{n_+\}} \hspace{-.5cm} \sum_{R_2^>\subseteq N_> -R_1\atop |R_2^>|\leq m-|L_1^<|}  \hspace{-.5cm} 
(q-1)^{|R_2^>|} q^{\wt_{R_1^>}^\uparrow(R_2^>)+(m-|L_1^<|-|R_2^>|)|R_1^>|}\vphi_{|R_2^>|}^{m-|L_1^<|}(q) \overrightarrow{\psi}_N^{R_2^>}
\\
&\hspace*{3cm}\odot \chi^{n_-\underset{m-|L_1^<|-|R_2^>|}{\smile}n_{+}}.
\end{align*}
We can plug this back into the original computation to get
\begin{align*}
\chi^{n_{--}\underset{m}{\smile} n_{++}}
&=q^{m\#\{n_-\}}\hspace{-.3cm}\sum_{L_1^<\subseteq N_<,R_1\subseteq N\atop |L_1^<|\leq m} (q-1)^{|L_1^<|} q^{|L_1^<|\#\{n_+\}}\vphi_{|L_1^<|}^{m}(q) \overleftarrow{\psi}_N^{L_1^<\hookleftarrow R_1}\\
&\hspace*{1cm}\odot q^{(m-|L_1^<|)\#\{n_+\}} \hspace{-.5cm} \sum_{R_2^>\subseteq N_> -R_1\atop |R_2^>|\leq m-|L_1^<|}  \hspace{-.5cm} 
(q-1)^{|R_2^>|} q^{\wt_{R_1^>}^\uparrow(R_2^>)+(m-|L_1^<|-|R_2^>|)|R_1^>|}\vphi_{|R_2^>|}^{m-|L_1^<|}(q) \overrightarrow{\psi}_N^{R_2^>}
\\
&\hspace*{2cm}\odot \chi^{n_-\underset{m-|L_1^<|-|R_2^>|}{\smile}n_{+}}\\
&= q^{m\#\{n_-,n_+\}}\hspace{-.3cm}\sum_{{L_1^<\subseteq N_<,R_1\subseteq N\atop R_2^>\subseteq N_> -R_1}\atop |L_1^<|+|R_2^>|\leq m} (q-1)^{|L_1^<|+|R_2^>|} q^{\wt_{R_1^>}^\uparrow(R_2^>)+(m-|L_1^<|-|R_2^>|)|R_1^>|}\vphi_{|L_1^<|+|R_2^>|}^{m}(q)\\
&\hspace*{1cm}\cdot \overleftarrow{\psi}_N^{L_1^<\hookleftarrow R_1}\odot \overrightarrow{\psi}_N^{R_2^>} \odot \chi^{n_-\underset{m-|L_1^<|-|R_2^>|}{\smile}n_{+}}\\
&= q^{m\#\{n_-,n_+\}}\hspace{-.6cm}\sum_{{L_1^<\subseteq N_<,R_1\subseteq N\atop L_2\subseteq N_\geq, R_2^>\subseteq N_> -R_1}\atop |L_1^<|+|R_2^>|\leq m} \hspace{-.6cm}(q-1)^{|L_1^<|+|R_2^>|} q^{\wt_{R_1^>}^\uparrow(R_2^>)+(m-|L_1^<|-|R_2^>|)|R_1^>|}\vphi_{|L_1^<|+|R_2^>|}^{m}(q)\\
&\hspace*{1cm}\cdot \overleftarrow{\psi}_N^{L_1^<\hookleftarrow R_1}\odot \overrightarrow{\psi}_N^{L_2\hookrightarrow R_2^>} \odot \chi^{n_-\underset{m-|L_1^<|-|R_2^>|}{\smile}n_{+}},
\end{align*}
where the last equality is by (\ref{V^KToV^hook}).  We simplify to get
\begin{align*}
\chi^{n_{--}\underset{m}{\smile} n_{++}}
&=q^{m\#\{n_-,n_+\}}\hspace{-.3cm}\sum_{
{L\subseteq N_\geq, R\subseteq N\atop F\subseteq N_<\cup (N_>-R)}\atop |F|\leq m}
(q-1)^{|F|} q^{\wt_{R\cap N_>}^\uparrow(F\cap N_>)+(m-|F|)|R\cap N_>|}\vphi_{|F|}^{m}(q)
\\
&\hspace*{1cm}\cdot \overleftarrow{\psi}_N^{F\cap N_<\hookleftarrow R}\odot \overrightarrow{\psi}_N^{L\hookrightarrow F\cap N_>} \odot \chi^{n_-\underset{m-|F|}{\smile}n_{+}}\\
&= q^{m\#\{n_-,n_+\}} \sum_{0\leq b\leq \min\{|N_<|,|N_>|\}\atop b\leq f\leq \min\{m,|N_<\cup N_>|\}}(q-1)^{f}q^{(m-f)b} \vphi_{f}^{m}(q)\psi^{(b;f)}_{N_=\subseteq N}\odot \chi^{n_-\underset{m-f}{\smile}n_+},
\end{align*}
as desired.
\end{proof}

\begin{remark}
If $N_<\cup N_>=\emptyset$, then Theorem \ref{TensorToPeel} reduces to the uninsightful
$$V^{n_{--}\underset{m+\ell}{\smile}n_{++}}\cong V^{n_{-}\underset{m+\ell}{\smile}n_{+}}.$$
On the other hand, $N_==\emptyset$ gives a family of equations 
$$
V^{n_{--}\underset{m}{\smile}n_{++}}\cong q^{mr}(q-1)^{m} \bigoplus_{0\leq b\leq f\leq m}  q^{(m-f)b}\vphi_f^m(q)V_{\emptyset\subseteq N}^{(b;f)}. 
$$
that each depend on relative order of $n_{+}=n_-$ in $N'$.  For example, if $N_\leq =\emptyset$ or $N_\geq =\emptyset$, then we recover Theorem \ref{RainbowRestriction}.  Alternatively, if $|N|=2\ell$ and $N_<$ is the first $\ell$ elements of $N$, then the coefficient of the full rainbow supercharacter 
$$\chi^\lambda=
\chi^{\begin{tikzpicture}[scale=.3,baseline=0cm]
	\foreach \x in {0,1,3,5,7,8}
		\node (\x) at (\x,0) [inner sep =-1pt] {$\bullet$};
	\node (4) at (4,0) [inner sep =-1pt] {$\circ$};
	\foreach \x in {2,6}
		\node at (\x,0) {$\scs\cdots$};
	\draw (0) to [out=-45,in=-135] (8);	
	\draw (1) to [out=-45,in=-135] (7);
	\draw (3) to [out=-45,in=-135] (5);
	\node at (4,.5) {$\scscs n_{-}=n_+$};	
\end{tikzpicture}}
$$
is 
$$(q-1)^{m}\sum_{k=\ell}^mq^{\nst_\lambda^\lambda+\binom{k-\ell}{2}}\vphi_k^m(q)\Pqbin{\ell}{k-\ell}{q}=(q-1)^{m} q^{\nst_\lambda^\lambda+(m-\ell)\ell}\vphi_\ell^m(q),$$
which simplifies to
$$  
\sum_{k=\ell}^m q^{\binom{k-\ell}{2}}\vphi_{k-\ell}^{m-\ell}(q)\Pqbin{\ell}{k-\ell}{q}= q^{(m-\ell)\ell}.
$$
\end{remark}

\subsection{The multirainbow case} \label{logModule}

For a sequence of subsets $N=N_1\supseteq N_2\supseteq \cdots N_k\supseteq \emptyset$, let 
\begin{equation}\label{FullOnionSetup}
N'=N\sqcup N_{\pm},\qquad \text{where} \qquad N_{\pm}=\bigcup_{j=1}^k \{n_j^-,n_j^+\},
\end{equation}
and for each $1\leq j\leq k$, $\{n_j^-,n_j^+\}\cup N_j$ is an interval in $N'$ with $\{n_j^-\}<N_j<\{n_j^+\}$.
For $\underline{b}=(b_1,\ldots, b_{k-1},0)$ and $\underline{f}=(f_1,\ldots, f_{k})$, a corresponding \emph{onion module} is given by
$$V_{N_k\subseteq  \cdots \subseteq N_1}^{(\underline{b};\underline{f})}=\Big(\bigotimes_{j=1}^{k-1} V_{N_{j+1}\subseteq N_j}^{(b_j;f_j)} \Big)\otimes V_{N_k}^{f_k}.$$

 For any sequence $(a_1,a_2,\ldots, a_k)$, let
$$a_{i\leq} =a_i+\cdots +a_k\qquad \text{and}\qquad a_{\leq i}=a_1+\cdots+a_i.$$
The iterated version of Theorem \ref{TensorToPeel} is as follows. 

\begin{corollary}  \label{OnionPeelingCorollary} Let $N= N_1\supset \cdots \supset N_{k}\supset \emptyset$  be a sequences of nested subsets with $N'$ as in (\ref{FullOnionSetup}).     
For $\underline{m}=(m_1,\ldots, m_k)\in \ZZ_{\geq 1}^k$, define the multiset 
$$\mu=\{n_1^-\underset{m_1}{\smile}n_1^+, n_2^-\underset{m_2}{\smile}n_2^+,\ldots, n_k^-\underset{m_k}{\smile}n_k^+\},$$
and the set
$$I_{\underline{m}}=\{(\underline{b},\underline{f})\in \ZZ_{\geq 0}^k\mid 0\leq b_j \leq f_j\leq m_{\leq j}-f_{\leq j-1}, 1\leq j\leq k, b_k=0\}.$$
Then
\begin{align*}
\Res_{\UT_N}^{\UT_{N'}}(V^\mu)\cong&
q^{\nst_{\epl{\mu}\cup\epr{\mu}}^\mu}(q-1)^{m_{1\leq}}\hspace{-.2cm}\bigoplus_{(\underline{b},\underline{f})\in I_{\underline{m}}}\hspace{-.2cm} \Big(\prod_{j=1}^{k} q^{(m_j-f_j)b_{j\leq}}\vphi_{f_j}^{m_{\leq j}-f_{\leq j-1}}(q) \Big)V_{N_k\subseteq  \cdots \subseteq N_1}^{(\underline{b};\underline{f})}.
\end{align*}
\end{corollary}

One basic application of this result is to decompose $\UT_N$-module
$$\CC\spanning\{\fkut_N\},$$
under left multiplication.  The decomposition into supercharacters is not obvious (unlike in the case of the regular module $\CC\spanning\{\UT_N\}$).  The following proposition implies that this module falls within the framework of Corollary \ref{OnionPeelingCorollary}. 

\begin{proposition} Let
$N'=N\cup \{n_-,n_1^+,\ldots, n_{|N|}^+\}$, where $n_-<n_1<n_1^+<n_2<\cdots < n_{|N|}<n_{|N|}^+$.  Then
$$\CC\spanning\{\fkut_N\}\cong (q-1)^{1-|N|}\Res_{\UT_N}^{\UT_{N'}}\Big(\bigotimes_{j=1}^{|N|-1} V^{n_-\smile n_j^+}\Big).$$
\end{proposition}
\begin{proof}
Let $\mu \in \scS_N$ and consider
$$\tr(u_\mu,\fkut_N)=\sum_{v\in \fkut_N} u_\mu v \bigg|_v=\#\{v\in \fkut_N\mid u_\mu v=v\}.$$
Since,
$$(u_\mu v)_{ik}=\left\{\begin{array}{ll} v_{ik}+v_{jk} & \text{if $i\smile j\in \mu$,}\\ v_{ik} & \text{otherwise,} \end{array}\right.$$
we may conclude 
$$\tr(u_\mu,\fkut_N)=q^{\binom{|N|}{2}} \prod_{i\smile j\in \mu} q^{-\wt_N^\uparrow(j)}.$$
On the other hand, by Proposition \ref{SupercharacterFormula},
$$\bigodot_{j=1}^{|N|-1} \chi^{n_-\smile n_j^+}(u_\mu)=(q-1)^{|N|-1}q^{\binom{|N|}{2}}\prod_{i\smile j\in \mu} q^{-\wt_N^\uparrow(j)}.$$
Comparing class function values gives the result.
\end{proof}

We can apply Corollary \ref{OnionPeelingCorollary} to the module
$$\bigotimes_{j=1}^{|N|} V^{n_-\smile n_j^+}$$
to decompose it into more reasonable modules.

\begin{corollary} For 
$$\bigotimes_{j=1}^{|N|} V^{n_-\smile n_j^+}\cong q^{\binom{|N|-1}{2}} (q-1)^{|N|}\bigoplus_{A\subseteq N} (q-1)^{|A|}\prod_{a\in A}[1+\wt_{N-A}^\uparrow(a)] V_N^A.$$
\end{corollary}
\begin{proof}
We prove a stronger statement by induction on $|N|$: for $m\geq 1$,
$$\chi^{n_-\underset{m}{\smile}n_{|N|}^+}\odot\bigodot_{j=1}^{|N|-1} \chi^{n_-\smile n_j^+} = q^{m(|N|-1)+\binom{|N|-1}{2}} (q-1)^{|N|+m-1}\bigoplus_{A\subseteq N} (q-1)^{|A|}\prod_{a\in A}[m+\wt_{N-A}^\uparrow(a)] \psi_N^A.$$ 
When $|N|=1$, then by Theorem \ref{RainbowRestriction},
$$\chi^{n_-\underset{m}{\smile}n_{|N|}^+}=(q-1)^{m+1}[m]\psi_N^1+(q-1)^{m+0}\psi_N^0,$$
as desired.  Assume $|N|>1$.  Then by Corollary \ref{OnionPeelingCorollary} and (\ref{DegenerateCase3}),
\begin{align*}
\chi^{n_-\underset{m}{\smile}n_{|N|}^+}\odot \bigodot_{j=1}^{|N|-1} \chi^{n_-\smile n_j^+} & = q^{m+n-1}(q-1)^2[m] \psi_N^{\{n_{|N|}\}}\odot \chi^{n_-\underset{m}{\smile}n_{|N|-1}^+} \odot \bigodot_{j=1}^{|N|-2} \chi^{n_-\smile n_j^+} \\
&+q^{m}[0]! \psi_N^{\emptyset} \odot \chi^{n_-\underset{m+1}{\smile}n_{|N|-1}^+} \odot \bigodot_{j=1}^{|N|-2} \chi^{n_-\smile n_j^+}.
\end{align*}
Apply induction to the first sum to get
\begin{align*}
&q^{m+n-1}(q-1)^2[m]\psi_N^{\{n_{|N|}\}}\odot \chi^{n_-\underset{m}{\smile}n_{|N|-1}^+} \odot \bigodot_{j=1}^{|N|-2} \chi^{n_-\smile n_j^+}\\
&=q^{m(|N|-1)+\binom{|N|-1}{2}}(q-1)^{|N|+m-1}[m] \psi_N^{\{n_{|N|}\}}\odot  \sum_{A\subseteq N}(q-1)^{|A|+1}\prod_{a\in A} [m+\wt_{N-A-\{n_{|N|}\}}(a)] \psi_N^A\\
&=q^{m(|N|-1)+\binom{|N|-1}{2}}(q-1)^{|N|+m-1}[m]   \sum_{\{n_{|N|}\}\subseteq A\subseteq N}(q-1)^{|A|}\prod_{a\in A} [m+\wt_{N-A}(a)] \psi_N^A.
\end{align*}
Apply induction to the second sum to obtain
\begin{align*}
q^{m}[0]! & \psi_N^{\emptyset} \odot \chi^{n_-\underset{m+1}{\smile}n_{|N|-1}^+} \odot \bigodot_{j=1}^{|N|-2} \chi^{n_-\smile n_j^+}\\
&=q^{m+(m+1)(|N|-2)+\binom{|N|-2}{2}}(q-1)^{|N|+m-1} \sum_{A\subseteq N-\{n_{|N|}\}}(q-1)^{|A|}\prod_{a\in A} [m+\wt_{N-A}(a)] \psi_N^A\\
&=q^{m(|N|-1)+\binom{|N|-1}{2}}(q-1)^{|N|+m-1} \sum_{A\subseteq N-\{n_{|N|}\}}(q-1)^{|A|}\prod_{a\in A} [m+\wt_{N-A}(a)] \psi_N^A.
\end{align*}
Set $m=1$ to get the main result.
\end{proof} 

We apply Proposition \ref{V^KDecomposition} to get a decomposition into supercharacters.

\begin{corollary}  If $n'$ is the maximal element in $N$, then
$$\tr(\cdot,\CC\spanning\{\fkut_N\})=q^{\binom{|N|-2}{2}} \sum_{\lambda\in \scS_{N}\atop n'\notin \epr{\lambda}} q^{\nst_\lambda^\lambda}\Big(\sum_{\epr{\lambda}\subseteq A\subseteq N-\{n'\}}  q^{\nst_{A-\epr{\lambda}}^\lambda}\prod_{a\in A} (q^{\wt_{N-A}^\uparrow(a)}-1)\Big)\chi^\lambda.$$
\end{corollary}

\subsection{The further decomposition into supercharacters}

This section returns to the problem of restricting to supercharacters.  The final decomposition is given in Corollary \ref{DoubleRainbowToSupercharacter}, below, which essentially unpacks Theorem \ref{TensorToPeel}.  However, we first need a result that takes Lemma \ref{LeftInterferenceRestriction} further into a decomposition of supercharacters.  Note that if $\nu=\emptyset$ in the following lemma, then (a) reduces to Proposition \ref{LeftRightEndpointModules} and (b) reduces to Corollary \ref{RainbowToSupercharacters}.

\begin{lemma} \label{InterferenceToSupercharacters}  Let $\{k_-,k_+\}\cup K\subset N$ be an interval with $\{k_-\}<K<\{k_+\}$.  Let $\nu\in \scM_N$ with $\nu_K=\emptyset$.  Then
\begin{enumerate}
\item[(a)] Let $\epr{X}=(\epr{\nu}\cap K)$. As a $\UT_K$-module
\begin{equation*}
V^\nu\otimes V_K^J\cong V^\nu\otimes\hspace{-.2cm} \bigoplus_{I\subseteq K-\epr{X}\atop J'\subseteq J, |J'|=|I|} \hspace{-.2cm}   \frac{q^{\wt_{\epr{X}}^{\uparrow}(I)+\wt_{\epr{X}}^\uparrow(J-J')+\wt_I^\uparrow(J-J')}}{q^{\wt_{J'}^\uparrow(J-J')}} V_K^{J'\hookleftarrow I}.
\end{equation*}
\item[(b)]  Suppose $\nu\in \scS_N$, and let $\epl{X}=(\epl{\nu}\cap K)$. As a $\UT_K$-module
$$V^\nu\otimes V^{k_-\underset{\ell}{\smile}k_{+}}\cong V^\nu\otimes (q-1)^{\ell} \hspace{-.3cm}\bigoplus_{{\lambda\in \scS_{K}\atop\lambda\cup\nu\in \scS_N}\atop |\lambda|\leq \ell} \hspace{-.3cm}
q^{\nst_\lambda^{\lambda\cup \nu}}
\sum_{l=|\lambda|}^{\ell}\frac{q^{(\ell-l)|\epl{X}|}}{q^{l|\nu'|}}\vphi_l^\ell(q)\Pqbin{\cP(\uncr{\lambda\cup\nu})}{l-|\lambda|\subseteq \bl_{\epr{K}}(\uncr{\lambda\cup \nu}),0}{q}V^\lambda,$$
 where $\nu'=\{i\smile j\in \nu\mid \{i\}< K<\{j\}\}$.
 \end{enumerate}
\end{lemma}
\begin{proof} (a) By Corollary \ref{V^KRealization}, as elements of $\scf(\UT_K)$,
$$(q-1)^{|J|}\psi_K^J=\Res_{\UT_K}^{\UT_{K\cup\{k_-,k_+\}}} \Big(\bigodot_{j\in J} \chi^{j\smile k_+}\Big).$$
The decomposition in Proposition \ref{LeftRightEndpointModules} implies that it suffices to show that if $\mu\in \scS_K$ with $\epl{\mu}\subseteq J$ and $\epr{\mu}\in K-\epr{X}$, then the coefficient of $\chi^\nu\odot\chi^\mu$ in 
\begin{equation}\label{ProofCoefficientFormula}
\chi^\nu\odot \Res_{\UT_K}^{\UT_{K'}} \Big(\bigodot_{j\in J} \chi^{j\smile k_+}\Big)
\qquad\text{is}\qquad
(q-1)^{|J|} \frac{q^{\nst_\mu^\mu+ \wt_{\epr{X}}^{\uparrow}(\epr{\mu})+\wt_{\epr{X}}^\uparrow(J-\epl{\mu})+\wt_{\epr{\mu}}^\uparrow(J-\epl{\mu})}}{q^{\wt_{\epl{\mu}}^\uparrow(J-\epl{\mu})}}.
\end{equation}
To prove (\ref{ProofCoefficientFormula}) we induct on $|J|$.  If $|J|=1$, then we apply the $\dag$-version of Lemma \ref{LeftInterferenceRestriction} to obtain
$$\chi^\nu\odot \chi^{j\smile k_+}=\chi^\nu\odot\Big((q-1)q^{\wt_{\epr{X}}^\uparrow(j)} \chi^\emptyset + (q-1)\sum_{j<l<k_+\atop l\notin \epr{X}} q^{\wt_{\epr{X}}^\uparrow(l)}\chi^{j\smile l}\Big),$$
and all the summands satisfy (\ref{ProofCoefficientFormula}).

Suppose $|J|>1$.  Let $j_0\in J$ be minimal.  Then by first restricting $\chi^{j_0\smile k_+}$, \begin{equation*}
\chi^\nu \odot  \Big(\bigodot_{j\in J} \chi^{j\smile k_+}\Big)
=\chi^\nu \odot(q-1)\Big(q^{\wt_{\epr{X}}^\uparrow(j_0)} \chi^\emptyset + \sum_{j_0<l<k_+\atop l\notin \epr{X}} q^{\wt_{\epr{X}}^\uparrow(l)}\chi^{j_0\smile l}\Big)\odot  \Big(\bigodot_{j\in J\atop j\neq j_0} \chi^{j\smile k_+}\Big).
\end{equation*}
Proposition \ref{SimpleRestrictionTensor} (b) implies that the set partition $\mu$ only appears in the decomposition of at most one of the middle terms of the sum.  That is,  
if $j_0\notin \epl{\mu}$ then $\chi^\nu\odot\chi^\mu$ only appears in the expansion of 
$$\chi^\nu \odot(q-1)q^{\wt_{\epr{X}}^\uparrow(j_0)} \chi^\emptyset\odot \Big(\bigodot_{j\in J\atop j\neq j_0} \chi^{j\smile k_+}\Big)$$
so by induction the coefficient of $\chi^\nu\odot\chi^\mu$ is
$$(q-1)q^{\wt_{\epr{X}}^\uparrow(j_0)} (q-1)^{|J|-1} \frac{q^{\nst_\mu^\mu+ \wt_{\epr{X}}^{\uparrow}(\epr{\mu})+\wt_{\epr{X}}^\uparrow(J-\epl{\mu}-\{j_0\})+\wt_{\epr{\mu}}^\uparrow(J-\epl{\mu}-\{j_0\})}}{q^{\wt_{\epl{\mu}}^\uparrow(J-\epl{\mu}-\{j_0\})}}.$$
The minimality of $j_0$ implies $\wt_{\epr{\mu}}^\uparrow(j_0)=\wt_{\epl{\mu}}^\uparrow(j_0)$, and so this expression simplifies to the desired term.

If, on the other hand, $j_0\smile l\in \mu$, then the coefficient $\chi^\nu\odot \chi^\mu$ only appears in the expansion of 
$$\chi^{\nu}\odot(q-1) q^{\wt_{\epr{X}}^\uparrow(l)}\chi^{j_0\smile l}\odot \Big(\bigodot_{j\in J\atop j\neq j_0} \chi^{j\smile k_+}\Big).$$
Reorganize slightly so that we can use induction to show that the coefficient of $\chi^{\nu\cup\{j_0\smile l\}}\odot \chi^{\mu-\{j_0\smile l\}}$ in 
$$(q-1) q^{\wt_{\epr{X}}^\uparrow(l)}\chi^{\nu\cup \{j_0\smile l\}}\odot  \Big(\bigodot_{j\in J\atop j\neq j_0} \chi^{j\smile k_+}\Big)$$
is
$$(q-1) q^{\wt_{\epr{X}}^\uparrow(l)}(q-1)^{|J|-1} \frac{q^{\nst_{\mu-\{j_0\smile l\}}^{\mu-\{j_0\smile l\}}+ \wt_{\epr{X}\cup\{l\}}^{\uparrow}(\epr{\mu}-l)+\wt_{\epr{X}\cup\{l\}}^\uparrow(J-\epl{\mu})+\wt_{\epr{\mu}-l}^\uparrow(J-\epl{\mu})}}{q^{\wt_{\epl{\mu}-j_0}^\uparrow(J-\epl{\mu})}},
$$
where the minimality of $j_0$ again gives us the desired coefficient.

(b) Since $\nu_K=\emptyset$, we can use Lemma \ref{LeftInterferenceRestriction} with $\epl{X}=\epl{\nu}\cap K$ and (a) with $\epr{X}=\epr{\nu}\cap K$ to decompose
\begin{align*}
\chi^\nu\odot\chi^{k_-\underset{\ell}{\smile}k_{+}} & = \chi^\nu\otimes\hspace{-.2cm} \sum_{L\subseteq K-\epl{X}\atop |L|\leq \ell} \hspace{-.2cm}  (q-1)^{\ell} q^{\wt_{\epl{X}}^\downarrow(L)+(\ell-|L|)|\epl{X}|} \vphi_{|L|}^{\ell}(q)   \psi_{K}^L\\
&=\sum_{L\subseteq K-\epl{X}\atop |L|\leq \ell} \hspace{-.2cm}   (q-1)^{\ell} q^{\wt_{\epl{X}}^\downarrow(L)+(\ell-|L|)|\epl{X}|} \vphi_{|L|}^{\ell}(q) \chi^\nu\\
&\hspace*{2cm} \odot \sum_{R\subseteq K-\epr{X}\atop L'\subset L, |L'|=|R|}  \frac{q^{\wt_{\epr{X}}^{\uparrow}(R)+\wt_{\epr{X}}^\uparrow(L-L')+\wt_R^\uparrow(L-L')}}{q^{\wt_{L'}^\uparrow(L-L')}} \psi_{K}^{L'\hookleftarrow R}.
\end{align*}
By (\ref{V^hookDecomposition}), Proposition \ref{LeftRightEndpointModules}, and then collecting terms, 
\begin{align}
\chi^\nu\odot\chi^{k_-\underset{\ell}{\smile}k_{+}} &=\sum_{L\subseteq K-\epl{X}\atop |L|\leq \ell} \hspace{-.2cm}   (q-1)^{\ell} q^{\wt_{\epl{X}}^\downarrow(L)+(\ell-|L|)|\epl{X}|} \vphi_{|L|}^{\ell}(q) \chi^\nu\notag
\\
&\hspace*{2cm} \odot \sum_{R\subseteq K-\epr{X}\atop L'\subseteq L, |L'|=|R|}  q^{\wt_{\epr{X}}^{\uparrow}(R)+\wt_{\epr{X}}^\uparrow(L-L')} \sum_{\lambda\in \scS_{K}\atop \epl{\lambda} =L', \epr{\lambda}=R}   q^{\nst_\lambda^\lambda+\nst_{L-\epl{\lambda}}^\lambda}\chi^\lambda \notag\\
 &=\chi^\nu\odot \sum_{{\lambda\in \scS_{K}\atop \lambda\cup\nu\in \scS_N}\atop |\lambda|\leq \ell} 
 q^{\wt_{\epl{X}}^\downarrow(\epl{\lambda})+\wt_{\epr{X}}^{\uparrow}(\epr{\lambda})+\nst_\lambda^\lambda}
 \sum_{l=|\lambda|}^{\ell}(q-1)^{\ell}q^{(\ell-l)|\epl{X}|} \vphi_{l}^\ell(q)
\notag \\
&\hspace*{2cm}\cdot \sum_{L\subseteq K-\epl{X}- \epl{\lambda}\atop |L|+|\lambda|=l} \hspace{-.2cm}   q^{\wt_{\epl{X}}^\downarrow(L)+\wt_{\epr{X}}^\uparrow(L)+\nst_{L}^\lambda}\chi^\lambda. \label{ProofA}
 \end{align}
 Note that
 $$\wt_{\epl{X}}^\downarrow(\epl{\lambda})+\wt_{\epr{X}}^{\uparrow}(\epr{\lambda})+\nst_\lambda^\lambda=\nst_\lambda^{\lambda\cup\tilde\nu},\qquad \text{where}\qquad \tilde\nu=\{i\smile j\in \nu\mid \{i,j\}\cap K\neq \emptyset\},$$
 and since $K$ is an interval,
 \begin{equation}\label{ProofB}\nst_\lambda^{\lambda\cup \nu}-\nst_\lambda^{\lambda\cup\tilde\nu}=|\lambda||\nu'|.
 \end{equation}
Fix a point $y\in K-\epl{X}-\epl{\lambda}$.  Then by (\ref{NstWtConversion}),
$$\nst_y^{\lambda\cup \nu}= \nst_y^\lambda+\wt^\downarrow_{\epl{\nu}}(y)-\wt^\downarrow_{\epr{\nu}}(y).$$
With respect to $y$, the elements of $i\in \epl{\nu}$ with $i<y$ fall into 4 categories:
\begin{enumerate}
\item[(1)] $i\smile k\in \nu$ with $k<y$,
\item[(2)] $i\smile k\in \nu$ with $y<k$ and $k\in K$,
\item[(3)] $i\smile k\in \nu$ with $y<k$ and $i\in K$,
\item[(4)] $i\smile k\in \nu$ with $y<k$ and $i,k\notin K$.
\end{enumerate}
Thus,
$$\wt^\downarrow_{\epl{\nu}}(y)-\wt^\downarrow_{\epr{\nu}}(y)=0+\wt^\uparrow_{\epr{X}}(y)+\wt^\downarrow_{\epl{X}}(y)+|\nu'|\cdot 1,$$
and we obtain
$$\sum_{L\subseteq K-\epl{X}- \epl{\lambda}\atop |L|+|\lambda|=l} \hspace{-.2cm}   q^{\wt_{\epl{X}}^\downarrow(L)+\wt_{\epr{X}}^\uparrow(L)+\nst_{L}^\lambda}=\sum_{L\subseteq K-\epl{X}- \epl{\lambda}\atop |L|+|\lambda|=l} \hspace{-.2cm}  q^{\nst_{L}^{\lambda\cup \nu}-\nst_L^{\nu'}}=\frac{1}{q^{(l-|\lambda|)|\nu'|}}\Pqbin{\cP(\uncr{\lambda\cup\nu})}{l-|\lambda|\subseteq \bl_{\epr{K}}(\uncr{\lambda\cup \nu}),0}{q}.$$
Plug this last observation and (\ref{ProofB}) into (\ref{ProofA}) to get the desired result.
\end{proof}

\begin{remark}
There is an inherent asymmetry in the construction in Theorem \ref{TensorToPeel}, but in theory that asymmetry should disappear again when we express the module in terms of supercharacters as in Lemma \ref{InterferenceToSupercharacters}.  In particular, the coefficients may be simpler if one uses the $\ract$-version of Lemma \ref{TensorToPeel} rather than the $\lact$-version given above.  Computational evidence suggests the coefficients have a slightly simpler symmetric expression. 
\end{remark}

However, from this asymmetry, we get the following corollary.
 
\begin{corollary} Let $K\subset N$ be an interval and let $\lambda,\nu\in \scS_N$ satisfy (a) $\lambda\cup \nu\in \scS_N$, (b) $\lambda_K=\lambda$, (c) $\nu=\{i\smile j\in \nu\mid |\{i,j\}\cap K|=1\}$.  Then for $\ell\in \ZZ_{\geq |\lambda|}$,
\begin{equation*}
 \sum_{l=|\lambda|}^{\ell}q^{(\ell-l)|\epl{\nu}\cap K|}\vphi^\ell_l(q)\Pqbin{\cP(\uncr{\lambda\cup\nu})}{l-|\lambda|\subseteq \bl_{\epr{K}}(\uncr{\lambda\cup \nu}),0}{q}= \sum_{l=|\lambda|}^{\ell}q^{(\ell-l)|\epr{\nu}\cap K|}\vphi^\ell_l(q)\Pqbin{\cP(\uncr{\lambda\cup\nu})}{l-|\lambda|\subseteq \bl_{\epl{K}}(\uncr{\lambda\cup \nu}),0}{q}.
 \end{equation*}
\end{corollary}
\begin{example}  
If, for example, $K=\{4,\ldots,12\}\subseteq N=\{1,2,\ldots,13 \}$, $\ell=3$,
$$\lambda=
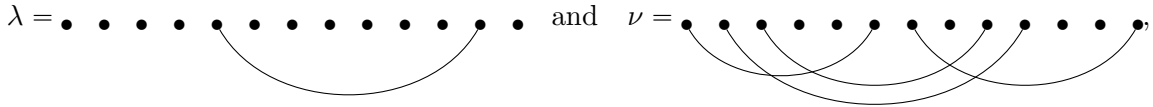
\begin{tikzpicture}[scale=.5,baseline=0]
	\foreach \x in {0,...,12}
		\node (\x) at (\x,0) [inner sep=-1pt] {$\bullet$};
	\draw (4) to [out=-60,in=-120]  (11);
\end{tikzpicture}\quad \text{and}\quad
\nu=\begin{tikzpicture}[scale=.5,baseline=0]
	\foreach \x in {0,...,12}
		\node (\x) at (\x,0)  [inner sep=-1pt] {$\bullet$};
	\draw (0) to [out=-60,in=-120]  (5);
	\draw (1) to [out=-60,in=-120]  (9);
	\draw (2) to [out=-60,in=-120]  (8);
	\draw (6) to [out=-60,in=-120]  (12);
\end{tikzpicture},$$
then this corollary says that
\begin{align*}
(q^3-1)&\Big(q^2+(q^2-1)q(q+q^2+3q^3+q^4) + (q^2-1)(q-1)(q^3+3q^4+4q^5+4q^6+3q^7)\Big)\\
&=(q^3-1)\Big(q^6+(q^2-1)q^3(3q^3+q^4)+(q^2-1)(q-1)(3q^6+3q^7)\Big).
\end{align*}
Vis-a-vis the remark, the ``correct" expression may be the even simpler
$$(q^3-1)\Big(q^{10}+(q^2-1)q^5(2q^3)+(q^2-1)(q-1)q^6\Big).$$
\end{example}

For $\gamma\in \scS_N$, let 
$$\gamma_=={}_=\gamma_=\qquad \text{and} \qquad\gamma_{\neq}=\gamma-\gamma_=,$$
and recall the notation (\ref{ArcBlocks}) that selects a subset of the arcs based on where their endpoints lie.  

\begin{corollary}\label{DoubleRainbowToSupercharacter}
\begin{align*}
V^\mu\cong q^{\nst_{\epl{\mu}\cup \epr{\mu}}^\mu}(q-1)^{m+\ell}&\bigoplus_{\gamma\in \scS_N}q^{\nst_\gamma^\gamma}\sum_{|\gamma_\neq|\leq f\leq m\atop |\gamma_=|\leq l\leq m-f+\ell}q^{(m-f-l)|{}_\leq\gamma_>|+\ell|{}_=\gamma_>|}\vphi_f^m(q)\vphi_l^{m-f+\ell}(q)\\
&\cdot\Pqbin{\cP(\uncr{\gamma})}{f-|\gamma_\neq|\subseteq \bl_{\epr{N_<}}(\uncr{\gamma})\cup \bl_{\epl{N_>}}(\uncr{\gamma}),l-|\gamma_=|\subseteq \bl_{\epr{N_=}}(\uncr{\gamma}),0}{q} V^\gamma.
\end{align*}
\end{corollary}

\begin{proof}
By Theorem \ref{TensorToPeel}, we have 
\begin{equation}\label{ProofFirstStep}
\chi^{n_{--}\underset{m}{\smile}n_{++}}\odot \chi^{n_-\underset{\ell}{\smile}n_{+}}=q^{\nst_{\epl{\mu}\cup \epr{\mu}}^\mu}\sum_{0\leq b\leq f\leq m} (q-1)^{f} q^{(m-f)b} \vphi_{f}^{m}(q)\psi_{N_=\subseteq N}^{(b;f)}\odot\chi^{n_-\underset{m-f+\ell}{\smile}n_{+}}.
\end{equation}
Use the decomposition in Proposition \ref{PeelDecomposition} to write
\begin{equation*}
\psi_{N_=\subseteq N}^{(b;f)}\odot\chi^{n_-\underset{m-f+\ell}{\smile}n_{+}} =\hspace{-.7cm} \sum_{\nu\in \scS_N,|\nu|\leq f\atop |{}_<\nu_{>}|=b,|{}_=\nu_=|=\emptyset}\hspace{-.7cm} q^{\nst_\nu^\nu} \Pqbin{\cP(\uncr{\nu})}{f-|\nu|\subseteq \bl_{\epr{N_<}}(\uncr{\nu})\cup \bl_{\epl{N_>}}(\uncr{\nu}),0}{q} \chi^\nu\odot\chi^{n_-\underset{m-f+\ell}{\smile}n_{+}}.
\end{equation*}
By Lemma \ref{InterferenceToSupercharacters} (b),
\begin{align*}
\chi^\nu\odot \chi^{n_-\underset{m-f+\ell}{\smile}n_{+}} &= \chi^\nu\odot \hspace{-.3cm} \sum_{{\lambda\in \scS_{N_=}\atop \lambda\cup\nu\in \scS_N}\atop |\lambda|\leq m-f+\ell}  \hspace{-.3cm}
 q^{\nst_\lambda^{\nu\cup\lambda}}
 \sum_{l=|\lambda|}^{m-f+\ell}(q-1)^{m-f+\ell}\frac{q^{(m-f+\ell-l)|\epl{\nu}\cap N_=|}}{q^{lb}}\vphi_{l}^{m-f+\ell}(q)\\
&\hspace*{1cm} \cdot \Pqbin{\cP(\uncr{\lambda\cup\nu})}{l-|\lambda|\subseteq \bl_{\epr{N_=}}(\uncr{\lambda\cup \nu}),0}{q}\chi^\lambda.
\end{align*}
Combine the two expressions by setting $\gamma=\lambda\cup \nu$, to get
\begin{align*}
\psi_{N_=\subseteq N}^{(b;f)}\odot\chi^{n_-\underset{m-f+\ell}{\smile}n_{+}} 
&=\hspace{-.5cm}\sum_{\gamma\in \scS_N,|\gamma_\neq|\leq f\atop |{}_<\gamma_>|=b,|\gamma|\leq m-f+\ell}
\hspace{-.5cm}q^{\nst_\gamma^\gamma}\sum_{l=|\gamma_=|}^{m-f+\ell}(q-1)^{m-f+\ell}\frac{q^{(m-f+\ell-l)|\epl{\nu}\cap N_=|}}{q^{lb}} \vphi_{l}^{m-f+\ell}(q)
\\
&\hspace*{1cm} \cdot\Pqbin{\cP(\uncr{\gamma})}{f-|\gamma_\neq|\subseteq \bl_{\epr{N_<}}(\uncr{\gamma})\cup \bl_{\epl{N_>}}(\uncr{\gamma}),l-|\gamma_=|\subseteq \bl_{\epr{N_=}}(\uncr{\gamma}),0}{q}\chi^\gamma
\end{align*}
Combine with (\ref{ProofFirstStep}) to get 
\begin{align*}
\chi^{n_{--}\underset{m}{\smile}n_{++}}&\odot \chi^{n_-\underset{\ell}{\smile}n_{+}}=
q^{\nst_{\epl{\mu}\cup \epr{\mu}}^\mu}(q-1)^{m+\ell}\sum_{\gamma\in \scS_N}q^{\nst_\gamma^\gamma}\sum_{|\gamma_\neq|\leq f\leq m\atop |\gamma_=|\leq l\leq m-f+\ell}q^{(m-f-l)|{}_<\gamma_>|+(m-f+\ell-l)|{}_=\gamma_>|}\\
&\hspace*{1cm}\cdot \vphi_{f}^{m}(q)\vphi_{l}^{m-f+\ell}(q)\Pqbin{\cP(\uncr{\gamma})}{f-|\gamma_\neq|\subseteq \bl_{\epr{N_<}}(\uncr{\gamma})\cup \bl_{\epl{N_>}}(\uncr{\gamma}),l-|\gamma_=|\subseteq \bl_{\epr{N_=}}(\uncr{\gamma}),0}{q} \chi^\gamma\\
&=q^{\nst_{\epl{\mu}\cup \epr{\mu}}^\mu}(q-1)^{m+\ell}\sum_{\gamma\in \scS_N}q^{\nst_\gamma^\gamma}\sum_{|\gamma_\neq|\leq f\leq m\atop |\gamma_=|\leq l\leq m-f+\ell}q^{(m-f-l)|{}_\leq\gamma_>|+\ell|{}_=\gamma_>|}\\
&\hspace*{1cm}\cdot \vphi_{f}^{m}(q)\vphi_{l}^{m-f+\ell}(q)\Pqbin{\cP(\uncr{\gamma})}{f-|\gamma_\neq|\subseteq \bl_{\epr{N_<}}(\uncr{\gamma})\cup \bl_{\epl{N_>}}(\uncr{\gamma}),l-|\gamma_=|\subseteq \bl_{\epr{N_=}}(\uncr{\gamma}),0}{q} \chi^\gamma,
\end{align*}
as desired.
\end{proof}

For the coefficient of the trivial character $\chi^\emptyset$, we have a more pleasing expression.

\begin{corollary}  The coefficient of $\chi^\emptyset\in \scf(\UT_{N})$ in the decomposition of $\chi^{n_{--}\underset{m}{\smile}n_{++}}\odot \chi^{n_-\underset{\ell}{\smile}n_{+}}$ is
\begin{equation*}
q^{2m}(q-1)^{m+\ell}\sum_{0\leq f\leq m\atop 0\leq l\leq m-f+\ell}\vphi_f^m(q)\vphi_l^{m-f+\ell}(q)\binom{|N-N_=|}{f}\binom{|N_=|}{l}.
\end{equation*}
\end{corollary}

\end{document}